\documentclass[12pt,draftcls,onecolumn]{IEEEtran}

\ifCLASSINFOpdf
\else
\fi

\usepackage{graphicx} 
\usepackage{times} 
\usepackage{amsmath} 
\usepackage{amssymb}  
\usepackage{amsthm}
\usepackage{amsfonts}
\usepackage{epstopdf}
\usepackage{subfig}
\usepackage{mathptmx} 
\usepackage{cite}
\usepackage[colorlinks=true]{hyperref}

\newtheorem{theorem}{Theorem}[section]
\newtheorem{lemma}[theorem]{Lemma}

\newtheorem{corollary}[theorem]{Corollary}

\newtheorem{remark}[theorem]{Remark}

\newtheorem{proposition}[theorem]{Proposition}

\newtheorem{problem}[theorem]{Problem}

\def\osc{\mathop{\mathrm{osc}}}  

\newcommand{\cB}{\cal B}

\newcommand{\sr}{\stackrel}

\newcommand{\tri}{\sr{\triangle}{=}}

\newcommand{\be}{\begin{equation}}
\newcommand{\ee}{\end{equation}}
\newcommand{\bea}{\begin{eqnarray}}
\newcommand{\eea}{\end{eqnarray}}
\newcommand{\bes}{\begin{eqnarray*}}
\newcommand{\ees}{\end{eqnarray*}}
\newcommand{\bi}{\begin{itemize}}
\newcommand{\ei}{\end{itemize}}
\newcommand{\ben}{\begin{enumerate}}
\newcommand{\een}{\end{enumerate}}
\newcommand{\bp}{\begin{problem}}
\newcommand{\ep}{\end{problem}}
\newcommand{\hso}{\hspace{.1in}}
\newcommand{\hst}{\hspace{.2in}}

\newcommand{\hse}{\hspace{.05in}}
\newcommand{\noi}{\noindent}

\begin{document}
%
\title{\bf Extremum Problems with Total Variation Distance and their Applications}

\author{Charalambos~D.~Charalambous, Ioannis~Tzortzis, Sergey Loyka and Themistoklis~Charalambous
\thanks{C. D. Charalambous and I. Tzortzis are with the Department of Electrical Engineering, University of Cyprus, Nicosia, Cyprus.
      Emails:  {\tt \{\small chadcha,tzortzis.ioannis\}@ucy.ac.cy}.}%
 \thanks{S. Loyka is with the School of Information Technology and Engineering, University of Ottawa, Ontario, Canada.
      Email:  {\tt\small sergey.loyka@ieee.org}.}%
\thanks{T. Charalambous is with the School of Electrical Engineering, Royal Institute of Technology (KTH), Stockholm, Sweden.
      Email:  {\tt\small themisc@kth.se}.}%
      }
\maketitle

%
%
%
%
\begin{abstract}
The aim of this paper is to investigate extremum problems with pay-off being the total variational distance metric defined on the space of probability measures, subject to linear functional constraints on the space of probability measures, and vice-versa; that is, with the roles of total variational metric and linear functional interchanged. Utilizing concepts from signed measures, the extremum probability measures of such problems are obtained in closed form, by identifying the partition of the support set and the mass of these extremum measures on the partition. The results are derived for abstract spaces; specifically,  complete separable metric spaces known as Polish spaces, while the high level ideas are also discussed for denumerable spaces endowed with the discrete topology. These extremum problems often arise in many areas, such as, approximating a family of  probability distributions by a given probability distribution, maximizing or minimizing entropy subject to total variational distance metric constraints, quantifying uncertainty of probability distributions by total variational distance metric, stochastic minimax control, and in many problems of information, decision theory, and minimax theory.

\noi {\bf Keywords:} Total variational distance, extremum probability measures, signed measures.
\end{abstract}

\IEEEpeerreviewmaketitle

%
%
%
%
\section{Introduction}\label{sec.Intro}

Total variational distance metric on the space of probability measures is a fundamental quantity in statistics and probability, which over the years appeared in many diverse applications. In information theory it is used to define strong typicality and asymptotic equipartition of sequences generated by sampling from a given distribution \cite{cover}. In decision problems, it arises naturally when discriminating the results of observation of two statistical hypotheses \cite{cover}. In studying the ergodicity of Markov Chains, it is used to define the Dobrushin coefficient and establish the contraction property of transition probability distributions \cite{Meyn93}. Moreover, distance in total variation of probability measures is related via upper and lower bounds to an anthology of distances and distance metrics \cite{gibbs}. The measure of distance in total variation of probability measures is a strong form of closeness of probability measures, and, convergence with respect to total variation of probability measures implies their convergence with respect to other distances and distance metrics.

In this paper, we formulate and solve several extremum problems involving the total variational distance metric and we discuss their applications. The main problems investigated are the following.
\begin{itemize}
\item[(a)] Extremum problems of linear functionals on the space of measures subject to a total variational distance metric constraint defined on the space of measures.
\item[(b)] Extremum problems of total variational distance metric on the space of measures subject to linear functionals on the space of measures.
\item[(c)] Applications of these extremum problems, and their relations to other problems.
\end{itemize}

The formulation of these extremum problems, their discussion in terms of applications, and the contributions of this paper are developed at the abstract level, in which systems are represented by probability distributions on abstract spaces (complete separable metric space, known as Polish spaces \cite{Dupuis97}), pay-offs are represented by linear functionals on the space of probability measures or by distance in variation of probability measures, and constraints by linear functionals or distance in variation of probability measures. We consider Polish spaces since they are general enough to handle various models of practical interest.

Utilizing concepts from signed measures, closed form expressions of the probability measures are derived which achieve the extremum of these problems. The construction of the extremum measures involves the identification of the partition of their support set, and their mass defined on these partitions. Throughout the derivations we make extensive use of lower and upper bounds of pay-offs which are achievable. Several simulations are carried out to illustrate the different features of the extremum solution of the various problems. An interesting observation concerning one of the extremum problems is its equivalent formulation as an extremum problem involving the oscillator semi-norm of the pay-off functional. The formulation and results obtained for these problems at the abstract level are discussed throughout the paper in the context of various applications, often assuming denumerable spaces endowed with the discrete topology. Some specific envisioned applications of the theory developed are listed below.
\begin{itemize}
\item[(i)] Dynamic Programming Under Uncertainty, to deal with  uncertainty of transition probability distributions, via minimax theory, with total variational distance metric uncertainty constraints to codify the impact of incorrect distribution models on performance of the optimal strategies \cite{cif2007}. This formulation is applicable to Markov decision problems subject to uncertainty.
\item[(ii)] Approximation of Probability Distributions with Total Variational Distance Metric,
to approximate a given probability distribution $\mu$ on a measurable space $(\Sigma,{\cal B}(\Sigma))$ by another distribution $\nu$ on $(\Sigma,{\cal B}(\Sigma))$, via minimization of the total variational distance metric between them subject to linear functional constraints. Model and graph reduction can be handled via such approximations.
\item[(iii)] Maximization or Minimization of Entropy Subject to Total Variational Distance Metric Constraints, to invoke insufficient reasoning based on maximizing the entropy $H(\nu)$ of an unknown probability distribution $\nu$ on denumerable space $\Sigma$  subject to a constraint on  the total variational distance metric.
\end{itemize}

The rest of the paper is organized as follows. In section \ref{sec.extr.prob.}, total variational distance is defined, the extremum problems are introduced, while several related problems are discussed together with their applications. In section \ref{char.extr.meas.}, some of the properties of the problems are discussed. In section \ref{subsec.Equiv.Extr.Prob}, signed measures are utilized to convert the extremum problems into equivalent ones, and to characterize the extremum measures on abstract spaces. In section \ref{cem}, closed form expressions of the extremum measures are derived for finite alphabet spaces. In section \ref{sec.rel.metr}, the relation between total variational distance and other distance metrics is discussed. Finally, in section \ref{ex} several examples are worked out to illustrate how the optimal solution of extremum problems behaves by examining different scenarios concerning the partition of the space $\Sigma$.

%
%
%
%
\section{Extremum Problems}
\label{sec.extr.prob.}

In this section, we will introduce the extremum problems we shall investigate. Let $(\Sigma,d_{\Sigma})$ denote a complete, separable metric space and $(\Sigma,{\cal B}(\Sigma))$ the corresponding measurable space, where ${\cal B}(\Sigma)$ is the $\sigma$-algebra generated by open sets in $\Sigma$. Let ${\cal M}_1(\Sigma)$ denote the set of probability measures on ${\cal B}(\Sigma)$. The total variational distance\footnote{The definition of total variation distance can be extended to signed measures.} is a metric \cite{dunford} $d_{TV}: {\cal M}_1(\Sigma) \times {\cal M}_1(\Sigma)\rightarrow [0,\infty)$  defined by
\vspace{-0.3cm}
\bea
d_{TV}(\alpha,\beta)\equiv||\alpha-\beta ||_{TV} \tri \sup_{P\in {\cal P}(\Sigma)} \sum_{F_i \in P} |\alpha (F_i)- \beta(F_i)| \; , \label{tveq1}
\eea
\noi where $\alpha, \beta \in {\cal M}_1(\Sigma)$ and ${\cal
P}(\Sigma)$ denotes the collection of all finite partitions of $\Sigma$. With respect to this metric, $({\cal M}_1(\Sigma),d_{TV})$ is a complete metric space. Since the elements of ${\cal M}_1(\Sigma)$ are probability measures, then $d_{TV}(\alpha,\beta)\leq 2$. In minimax problems one can introduce an uncertainty set based on distance in variation as follows. Suppose the probability measure ${\nu}\in {\cal M}_1(\Sigma)$ is unknown, while modeling techniques give  access to a nominal  probability measure ${ \mu} \in {\cal M}_1(\Sigma)$. Having constructed the  nominal probability measure, one may construct from empirical data, the  distance of the two measures with respect to the total variational distance  $||{ \nu} - { \mu}||_{TV}$. This will provide an estimate of the radius $R$, such that $||{ \nu} - { \mu}||_{TV} \leq R$, and hence characterize the set of all possible true measures  ${\nu}\in {\cal M}_1(\Sigma)$, centered at the nominal distribution ${ \mu} \in {\cal M}_1(\Sigma)$, and lying within the ball of radius $R$, with respect to the total variational distance $||\cdot||_{TV}$. Such a procedure is used in information theory to define strong typicality of sequences. Unlike other distances used in the past such as relative entropy \cite{pra96,Ugrinovskii,Petersen,nc2007,Charalambous07}, quantifying uncertainty via the metric $||\cdot||_{TV}$ does not require absolute continuity of measures\footnote{$\nu\in{\cal M}_1(\Sigma)$ is absolutely continuous with respect to $\mu\in{\cal M}_1(\Sigma)$, denoted by $\nu< < \mu$, if $\mu(A)=0$ for some $A\in {\cal B}(\Sigma)$ then $\nu(A)=0$.}, i.e., singular measures are admissible, and hence $\nu$ and $\mu$ need not be defined on the same space. Thus, the support set of $\mu$ may be $\tilde{\Sigma}\subset\Sigma$, hence $\mu(\Sigma\setminus\tilde{\Sigma})=0$ but $\nu(\Sigma\setminus\tilde{\Sigma})\neq0$ is allowed. For measures induced by stochastic differential equations (SDE's), variational distance uncertainty set models situations in which both the drift and diffusion coefficient of SDE's are unknown.

Define the spaces
\begin{align*}
&BC(\Sigma)\tri\left\{\ell:\Sigma\mapsto{\mathbb R}:\ell \hse\mbox{are bounded continuous}\right\},\\
&BM(\Sigma)\tri\left\{\ell:\Sigma\mapsto{\mathbb R}:\ell \hse\mbox{are bounded measurable functions}\right\},\\
&BC^+(\Sigma)\tri\left\{BC(\Sigma):\ell\geq 0\right\},\hso BM^+(\Sigma)\tri\left\{BM(\Sigma):\ell\geq 0\right\}.
\end{align*}
\noi $BC(\Sigma)$ and $BM(\Sigma)$ endowed with the sup norm $||\ell||\tri \sup_{x \in {\Sigma}} |\ell(x)|$, are Banach spaces \cite{dunford}. Next, we introduce the two main extremum problems we shall investigate in this paper.

\begin{problem}\label{problem1}
Given a fixed nominal distribution ${ \mu} \in   {\cal M}_1(\Sigma)$ and a parameter $R\in [0,2] $, define the class of true  distributions by
\bea
{\mathbb B }_R({ \mu})   \tri \Big\{ {\bf \nu} \in {\cal M}_1(\Sigma): ||{ \nu} -{ \mu}||_{TV} \leq R \Big\}, \label{prob1.dfn.eq.1}
\eea
and the average  pay-off with respect to the true probability measure ${ \nu} \in {\mathbb B }_R({ \mu})$ by
\vspace{-0.1cm}
\bea
{\mathbb L}_1( { \nu}) \tri \int_{\Sigma} \ell(x) \nu(dx),\hst \ell\in BC^+(\Sigma) \hso{or} \hso BM^+(\Sigma).  \label{prob1.dfn.eq.2}
\eea
 The objective is to find the extremum of the pay-off
 \bea
D^+(R) \tri    \sup_{ { \nu} \in  {\mathbb B}_R({ \mu})}  \int_{ \Sigma} \ell(x)\nu(dx).  \label{prob1.dfn.eq.3}
\eea
\end{problem}

Problem \ref{problem1} is a convex optimization problem on the space of probability measures. Note that, $BC^+(\Sigma)$, $BM^+(\Sigma)$ can be generalized to $L^{\infty,+}(\Sigma, {\cal B}(\Sigma), \nu)$,  the set of all ${\cal B}(\Sigma)$-measurable, non-negative  essentially bounded functions defined $\nu-a.e.$ endowed with the essential supremum norm   $||\ell||_{\infty,\nu} = \nu \mbox{-} \mbox{ess } \sup_{x \in \Sigma} \ell(x) \tri \inf_{\Delta \in {\cal N_{\eta}}} \sup_{x \in \Delta^c} \| \ell(x) \|$, where ${\cal N_{\nu}} = \small{\{A \in  {\cal B}(\Sigma): \nu(A) = 0\}}$.

In the context of minimax theory, Problem \ref{problem1} is important in uncertain stochastic control, estimation, and decision, formulated via minimax optimization. Such formulations are found in \cite{pra96,Ugrinovskii,Petersen,nc2007,Charalambous07} utilizing relative entropy uncertainty, and in \cite{Poor80,Vastola84} utilizing $L_1$ distance uncertainty. In the context of dynamic programming this is discussed in \cite{ctc2012}. The second extremum problem is defined below.

\begin{problem}\label{problem2}
Given a fixed nominal distribution ${ \mu} \in   {\cal M}_1(\Sigma)$ and a parameter $D\in [0,\infty)$, define the class of true  distributions by
\bea
{\mathbb Q }(D)   \tri \Big\{ {\bf \nu} \in {\cal M}_1(\Sigma):  \int_{\Sigma}\ell(x) \nu(dx) \leq D \Big\},\hst \ell\in BC^+(\Sigma) \hso{or} \hso BM^+(\Sigma), \label{prob2.dfn.eq.1}
\eea
and the total variation pay-off with respect to the true probability measure ${ \nu} \in {\mathbb Q }(D)$ by
\vspace{-0.1cm}
\bea
{\mathbb L}_2( { \nu}) \tri ||\nu-\mu||_{TV}.  \label{prob2.dfn.eq.2}
\eea
 The objective is to find the extremum of the pay-off
 \bea
R^-(D) \tri    \inf_{ { \nu} \in  {\mathbb Q}(D)}  ||\nu-\mu||_{TV},  \label{prob2.dfn.eq.3}
\eea
\noi whenever \footnote{If $\int_{\Sigma}\ell(x)\mu(dx)\leq D$ then $\nu^*=\mu$ is the trivial extremum measure of (\ref{prob2.dfn.eq.3}).}$\int_{\Sigma}\ell(x)\mu(dx)>D$.
\end{problem}

\noi Problem \ref{problem2} is important in the context of approximation theory, since distance in variation is a measure of proximity of two probability distributions subject to constraints. It is also important in spectral measure or density approximation as follows. Recall that a function $\{R(\tau):-\infty\leq \tau\leq\infty\}$ is the covariance function of a quadratic mean continuous and wide-sense stationary process if and only if it is of the form \cite{Wong85}
\vspace{-0.25cm}
\bes R(\tau)=\int_{-\infty}^{\infty}e^{2\pi\nu\tau}F(d\nu),
\ees
\noi where $F(\cdot)$ is a finite Borel measure on ${\mathbb R}$, called spectral measure. Thus, by proper normalization of $F(\cdot)$ via $F_N(d\nu)\tri\frac{1}{R(0)}F(d\nu)$, then $F_N(d\nu)$ is a probability measure on $\cB({\mathbb R})$, and hence Problem \ref{problem2} can be used to approximate the class of spectral measures which satisfy moment estimates. Spectral estimation problems are discussed extensively in \cite{Ferrante08,georgiou06,Ferrante07,Georgiou03kullback-leiblerapproximation,Pavon06}, utilizing relative entropy and Hellinger distances. However, in these references, the approximated spectral density is absolutely continuous with respect to the nominal spectral density; hence, it can not deal with reduced order approximation. In this respect, distance in total variation between spectral measures is very attractive.

\vspace{-0.1cm}
\subsection{Related Extremum Problems}
\label{subsec.Rel.Extr.Prob}
\noi Problems \ref{problem1}, \ref{problem2} are related to additional  extremum problems which are introduced below.
\begin{itemize}
\item[(1)] The solution of (\ref{prob1.dfn.eq.3}) gives the solution to the problem defined by \bea \label{primal1}R^+(D)\tri\sup_{\nu\in {\cal M}_1(\Sigma):\int_{\Sigma}\ell(x)\nu(dx)\leq D}||\nu-\mu||_{TV}.\eea
Specifically, $R^+(D)$ is the inverse mapping of $D^+(R)$. $D^+(R)$ is investigated in \cite{fcn2012} in the context of minimax stochastic control under uncertainty, following an alternative approach which utilizes large deviation theory to express the extremum measure by a convex combination of a tilted and the nominal probability measures. The two disadvantages of the method pursued in \cite{Ugrinovskii,Petersen,nc2007,Charalambous07} are  the following. 1) No explicit closed form expression for the extremum measure is given, and as a consequence, 2) its application to dynamic programming is restricted to a class of uncertain probability measures which are absolutely continuous with respect to the nominal measure $\mu(\Sigma)\in{\cal M}_1(\Sigma)$.

\item[ (2)] Let $\nu$ and $\mu$ be absolutely continuous with respect to the Lebesgue measure so that $\varphi(x)\tri\frac{d\nu}{d(x)}(x)$, $\psi(x)\tri\frac{d\mu}{dx}(x)$ (e.g., $\varphi(\cdot)$, $\psi(\cdot)$ are the probability density functions of $\nu(\cdot)$ and $\mu(\cdot)$, respectively. Then, $||\nu-\mu||_{TV}=\int_{\Sigma}|\varphi(x)-\psi(x)|dx$ and hence, (\ref{prob1.dfn.eq.3}) and (\ref{primal1}) are $L_1$-distance optimization problems.

\item[(3)] Let $\Sigma$ be a non-empty denumerable set endowed with the discrete topology including finite cardinality  $|\Sigma|$, with ${\cal M}_1(\Sigma)$ identified with the standard probability simplex in ${\mathbb R}^{|\Sigma|}$, that is, the set of all $|\Sigma|$-dimensional vectors which are probability vectors,  and   $\ell(x)\tri -\log\nu(x), x \in \Sigma$, where $\{\nu(x):x\in\Sigma \}\in{\cal M}_1(\Sigma)$, $\{\mu(x):x\in\Sigma \}\in{\cal M}_1(\Sigma)$.
 Then (\ref{prob1.dfn.eq.3}) is equivalent to maximizing the entropy of $\{\nu(x):x\in\Sigma \}$ subject to total variational distance metric constraint  defined by \bea \label{primal2}D^+(R)=\sup_{\nu\in{\cal M}_1(\Sigma):\sum_{x\in\Sigma}|\nu(x)-\mu(x)|\leq R}H(\nu).\eea
Problem (\ref{primal2}) is of interest when the concept of insufficient reasoning (e.g., Jayne's maximum entropy principle \cite{jaynes1,jaynes2}) is applied to construct a model for $\nu\in{\cal M}_1(\Sigma)$, subject to information quantified via total variational distance metric between $\nu$ and an empirical distribution $\mu$. In the context of stochastic uncertain control systems, and its relation to robustness, Problem (\ref{primal2}) with the total variational distance constraint replaced by relative entropy distance constraint is investigated in \cite{baras,rezcharahm12}.

\item[ (4) ] The solution of (\ref{prob2.dfn.eq.3}) gives the solution to the problem defined by \bea D^-(R)\tri\inf_{\nu\in{\cal M}_1(\Sigma):||\nu-\mu||_{TV}\leq R}\int_{\Sigma}\ell(x)\nu(dx)\label{rep4}.\eea
\noi Problems (\ref{prob2.dfn.eq.3}) and (\ref{rep4}) are important in approximating a class of probability distributions or spectral measures by reduced ones. In fact, the solution of (\ref{rep4}) is obtained precisely as that of Problem \ref{problem1}, with a reverse computation of the partition of the space $\Sigma$ and the mass of the extremum measure on the partition moving in the opposite direction.
\end{itemize}

%
%
%
%
\section{Characterization of Extremum Measures on Abstract Spaces}
\label{char.extr.meas.}

This section utilizes signed measures and some of their properties to convert Problems \ref{problem1}, \ref{problem2} into equivalent extremum problems. First, we discuss some of the properties of these extremum Problems.

\begin{lemma}\label{properties}
\noi
\begin{itemize}
\item[(1)] $D^+(R)$ is a non-decreasing concave function of $R$, and \bea D^+(R)=\sup_{||\nu-\mu||_{TV}=R}\int_{\Sigma}\ell(x)\nu(dx),\hst\mbox{if}\hst R\leq R_{\max},\label{prop1}\eea
    where $R_{\max}$ is the smallest non-negative number belonging to $[0,2]$ such that $D^+(R)$ is constant in $[R_{\max},2]$.
\item[(2)] $R^-(D)$ is a non-increasing convex function of $D$, and \bea R^-(D)=\inf_{\int_{\Sigma}\ell(x)\nu(dx)=D}||\nu-\mu||_{TV},\hst\mbox{if}\hst D\leq D_{\max},\label{prop2}\eea where $D_{\max}$ is the smallest non-negative number belonging to $[0,\infty)$ such that $R^-(D)=0$ for any $D\in[D_{\max},\infty)$.
\end{itemize}
\end{lemma}

\begin{proof}
\noi \textbf{(1)} Suppose $0\leq R_1\leq R_2$, then for every $\nu\in {\mathbb B}_{R_1}({ \mu})$ we have $||\nu-\mu||_{TV}\leq R_1\leq R_2$,
\noi and therefore $\nu\in {\mathbb B}_{R_2}({ \mu})$, hence
\bes \sup_{\nu\in {\mathbb B}_{R_1}({ \mu})}\int_{\Sigma}\ell(x)\nu(dx)\leq \sup_{\nu\in {\mathbb B}_{R_2}({ \mu})}\int_{\Sigma}\ell(x)\nu(dx),\ees
\noi which is equivalent to $D^+(R_1)\leq D^+(R_2)$. So $D^+(R)$ is a non-decreasing function of $R$. Now consider two points $(R_1,D^+(R_1))$ and $(R_2,D^+(R_2))$ on the linear functional curve, such that $\nu_1\in{\mathbb B }_{R_1}({ \mu})$ achieves the supremum of (\ref{prob1.dfn.eq.3}) for $R_1$, and $\nu_2\in{\mathbb B }_{R_2}({ \mu})$ achieves the supremum of (\ref{prob1.dfn.eq.3}) for $R_2$. Then,
$||\nu_1-\mu||_{TV}\leq R_1$ and $||\nu_2-\mu||_{TV}\leq R_2$. For any $\lambda\in(0,1)$, we have
\begin{align*}
||\lambda\nu_1+(1-\lambda)\nu_2-\mu||_{TV}\leq \lambda||\nu_1-\mu||_{TV}+(1-\lambda)||\nu_2-\mu||_{TV}\leq\lambda R_1+(1-\lambda)R_2=R.
\end{align*}
\noi Define $\nu^*\tri \lambda\nu_1+(1-\lambda)\nu_2$, $R\tri \lambda R_1+(1-\lambda)R_2$. The previous equation implies that $\nu^*\in{\mathbb B }_{R}({ \mu})$, hence $D^+(\lambda R_1+(1-\lambda)R_2)\geq \int_{\Sigma}\ell(x)\nu^*(dx)$. Therefore,
\begin{align*}
D^+(R)&=\sup_{\nu\in{\mathbb B }_R({ \mu})}\int_{\Sigma}\ell(x)\nu(dx)\geq \int_{\Sigma}\ell(x)\nu^*(dx)=\int_{\Sigma}\ell(x)\left(\lambda\nu_1(dx)+(1-\lambda)\nu_2(dx)\right)\\
&=\lambda \int_{\Sigma}\ell(x)\nu_1(dx)+(1-\lambda)\int_{\Sigma}\ell(x)\nu_2(dx)=\lambda D^+(R_1)+(1-\lambda)D^+(R_2).
\end{align*}

\noi So, $D^+(R)$ is a concave function of $R$. Also the right side of (\ref{prop1}), say $\bar{D}^+(R)$, is concave function of $R$. But $D^+(R)=\sup_{R'\leq R}\bar{D}^+(R')$ which completes the derivation of (\ref{prop1}).

\noi \textbf{(2)} Suppose $0\leq D_1\leq D_2$, then ${\mathbb Q}(D_1)\subset{\mathbb Q}(D_2)$, and $\inf_{\nu\in{\mathbb Q}(D_1)}||\nu-\mu||_{TV}\geq \inf_{\nu\in{\mathbb Q}(D_2)}||\nu-\mu||_{TV}$ which is equivalent to $R^-(D_1)\geq R^-(D_2)$. Hence, $R^-(D)$ is a non-increasing function of D. Now consider two points $(D_1,R^-(D_1))$ and $(D_2,R^-(D_2))$ on the total variation curve. Let $D\tri\lambda D_1+(1-\lambda)D_2$, $\nu^*\tri\lambda\nu_1+(1-\lambda)\nu_2$ and $\nu_1\in{\mathbb Q}(D_1)$, $\nu_2\in{\mathbb Q}(D_2)$ such that $||\nu_1-\mu||_{TV}=R^-(D_1)$ and $||\nu_2-\mu||_{TV}=R^-(D_2)$. Then, $\int_{\Sigma}\ell(x)\nu_1(dx)\leq D_1$ and $\int_{\Sigma}\ell(x)\nu_2(dx)\leq D_2$. Taking convex combination leads to
\begin{align*} &\lambda\int_{\Sigma}\ell(x)\nu_1(dx)+(1-\lambda)\int_{\Sigma}\ell(x)\nu_2(dx)\leq \lambda D_1+(1-\lambda)D_2=D,\end{align*}
\noi and hence $\nu^*\in {\mathbb Q}(D)$. So,
\begin{align*}
R^-(D)&=\inf_{\nu\in{\mathbb Q}(D)}||\nu-\mu||_{TV}\leq ||\nu^*-\mu||_{TV}=||\lambda\nu_1+(1-\lambda)\nu_2-\mu||_{TV}\\
&\leq \lambda||\nu_1-\mu||_{TV}+(1-\lambda)||\nu_2-\mu||_{TV}=\lambda R^-(D_1)+(1-\lambda)R^-(D_2).
\end{align*}

\noi This shows that $R^-(D)$ is convex function of $D$. Also the right side of (\ref{prop2}), say $\bar{R}^-(D)$, is convex function of $D$. But, $R^-(D)=\inf_{D'\leq D}\bar{R}^-(D')$ which completes the derivation of (\ref{prop2}).
\end{proof}

Let ${\cal M}_{sm}(\Sigma)$ denote the set of finite signed measures. Then, any ${ \eta} \in  {\cal M}_{sm}(\Sigma)$ has a Jordan decomposition \cite{halmos} $\big\{  { \eta}^+, { \eta}^-\big\}$ such that  ${ \eta }= { \eta}^+ - { \eta}^-$, and the total variation of ${ \eta}$ is defined by $|| { \eta}||_{TV} \tri { \eta}^+(\Sigma) + { \eta}^-(\Sigma)$.  Define the following subset ${\mathbb M}_0(\Sigma) \tri \Big\{ { \eta} \in {\cal M}_{sm}(\Sigma): { \eta}(\Sigma)=0\Big\}$. For ${ \xi} \in {\mathbb M}_0(\Sigma)$, then $\xi(\Sigma)=0$, which implies that ${ \xi}^+(\Sigma)= { \xi}^-(\Sigma)$, and hence ${ \xi}^+(\Sigma)= { \xi}^-(\Sigma)=\frac{||{  \xi}||_{TV}}{2}$. Then, ${ \xi}\tri { \nu} -{ \mu} \in  {\mathbb M}_0(\Sigma)$ and hence ${ \xi} = ({ \nu}-{ \mu})^+ -({ \nu}- { \mu})^- \equiv { \xi}^+- { \xi}^-$.

\vspace{-0.2cm}
\subsection{Equivalent Extremum Problem of $D^+(R)$}
\label{subsec.Equiv.Extr.Prob}

\noi Consider the pay-off of Problem \ref{problem1}, for ${\bf \ell} \in BC^+(\Sigma)$. Then the following inequalities hold.
\begin{align}
D^+(R)\tri  \int_{\Sigma} \ell(x) \nu(dx)&\overset{(a)}=\int_{\Sigma} \ell(x) \left(\xi^{+}(dx)- \xi^{-}(dx)\right) + \int_{\Sigma} \ell(x) \mu(dx) \nonumber  \\
& \overset{(b)}\leq \sup_{x\in \Sigma} \ell(x){ \xi}^{+}(\Sigma) - \inf_{x\in \Sigma} \ell(x) { \xi}^-(\Sigma) + \int_{\Sigma} \ell(x) \mu(dx) \nonumber  \\
& \overset{(c)}= \sup_{x\in \Sigma} \ell(x)\frac{||{ \xi} ||_{TV}}{2} - \inf_{x\in \Sigma} \ell(x)\frac{||{ \xi} ||_{TV}}{2} + \int_{\Sigma} \ell(x) \mu(dx) \nonumber  \\
&=\left\{ \sup_{x\in \Sigma} \ell(x) - \inf_{x\in \Sigma} \ell(x) \right\}\frac{||{ \xi} ||_{TV}}{2} + \int_{\Sigma} \ell(x) \mu(dx), \label{f2}
\end{align}
\noi where (a) follows by adding and subtracting $\int\ell d\mu$, and from the Jordan decomposition of $(\nu-\mu)$, (b) follows due to $\ell\in BC^+(\Sigma)$, (c) follows because any $\xi\in {\mathbb M}_0(\Sigma)$ satisfies $\xi^+(\Sigma)=\xi^-(\Sigma)=\frac{1}{2}||\xi||_{TV}$. For a given ${ \mu} \in   {\cal M}_1({\Sigma})$ and $\nu\in{\mathbb B }_R({ \mu})$ define the set
\bes
\widetilde{{\mathbb B }}_R({ \mu}) \tri \left\{ { \xi} \in {\mathbb M}_0(\Sigma): { \xi} = { \nu}- { \mu}, { \nu } \in {\cal M}_1(\Sigma), ||{ \xi}||_{TV} \leq R\right\}. \label{sb1}
\ees

\noi The upper bound in the right hand side of (\ref{f2}) is achieved by ${ \xi}^* \in \widetilde{{\mathbb B }}_R({ \mu})$ as follows. Let
\begin{align*}
&x^0 \in \Sigma^0  \tri  \left\{ x \in \Sigma: \ell(x) = \sup \{\ell(x): x\in \Sigma\} \equiv M \right\}, \\
&x_0 \in \Sigma_0  \tri  \left\{ x \in \Sigma: \ell(x) = \inf\{\ell(x): x\in \Sigma\} \equiv m \right\}.
\end{align*}
Take
\vspace{-0.25cm}
\bea
{ \xi}^*(dx) ={ \nu}^*(dx)-{ \mu}(dx) = \frac{R}{2} \left(\delta_{x^0}(dx) - \delta_{x_0}(dx)\right), \label{nm1}
\eea

\noi where $\delta_y(dx)$ denotes the Dirac measure concentrated at $y\in \Sigma$. This is indeed a signed measure with total variation $||\xi^*||_{TV}=||{ \nu}^*-{ \mu}||_{TV}=R$, and $\int_{\Sigma} \ell(x) ({\nu}^*- { \mu})(dx)=\frac{R}{2}\left(M-m\right)$. Hence, by using (\ref{nm1}) as a candidate of the maximizing distribution then the extremum Problem \ref{problem1} is equivalent to
\bea D^+(R)=\int_{\Sigma} \ell(x) \nu^*(dx) =  \frac{R}{2}  \left\{ \sup_{x \in {\Sigma}} \ell(x)  -  \inf_{x \in \Sigma} \ell(x) \right\}    + E_{\mu}(\ell), \label{f2n}
\eea
where ${ \nu}^*$  satisfies the constraint $||{ \xi}^*||_{TV}= ||{ \nu}^*-{ \mu}||_{TV} =R$, it is normalized  ${ \nu}^*(\Sigma)=1$, and $ 0 \leq \nu^*(A) \leq 1$ on any $A \in {\cal B}(\Sigma)$. Alternatively, the pay-off $\int_{\Sigma}\ell(x)\nu^*(dx)$ can be written as
\begin{align}
\int_{\Sigma}\ell(x)\nu^*(dx)&= \int_{\Sigma^0}     M   \nu^*(dx) +   \int_{\Sigma_0}   m     \nu^*(dx)  +   \int_{\Sigma \setminus \Sigma^0\cup \Sigma_0} \ell(x) \mu(dx).  \label{n2}
 \end{align}
\noi Hence, the optimal distribution ${\nu}^*\in {\mathbb B }_R({ \mu})$ satisfies
\begin{align}
&\int_{\Sigma^0} \nu^*(dx) = \mu(\Sigma^0) + \frac{R}{2}\in[0,1], \hso  \int_{\Sigma_0} \nu^*(dx) = \mu(\Sigma_0) - \frac{R}{2}\in[0,1],\nonumber\\
&\nu^*(A)= \mu(A), \hso \forall A \subseteq \Sigma \setminus \Sigma^0\cup \Sigma_0. \label{f322n}
\end{align}

\begin{remark}
\noi
\begin{itemize}
\item[(1)] For $\mu\in {\cal M}_1(\Sigma)$ which do not include point mass, and for $f\in BC^+(\Sigma)$, if $\Sigma^0$ and $\Sigma_0$ are countable, then (\ref{f322n}) is $\mu(\Sigma^0)=\mu(\Sigma_0)=0$, $\nu^*(\Sigma_0)=0$, $\nu^*(\Sigma^0)=\frac{R}{2}$, $\nu^*(\Sigma\setminus \Sigma^0\cup\Sigma_0)=\mu(\Sigma\setminus \Sigma^0\cup\Sigma_0)-\frac{R}{2}$.
\item[(2)] The first right side term in (\ref{f2n}) is related to the oscillator seminorm of $f\in BM(\Sigma)$ called global modulus of continuity, defined by $\osc (f)\triangleq \sup_{(x,y)\in\Sigma\times\Sigma}|f(x)-f(y)| =2\inf_{\alpha\in \mathbb{R}}||f-\alpha||$.
For $f\in BM^+(\Sigma)$, $\osc (f)=\sup_{x\in\Sigma}|f(x)|-\inf_{x\in\Sigma}|f(x)|$.
\end{itemize}
\end{remark}

\subsection{Equivalent Extremum Problem of $R^-(D)$}
\label{subsec.Equiv.Extr.Prob1}
\noi Next, we proceed with the abstract formulation of Problem \ref{problem2}. Consider the constraint of Problem \ref{problem2}, for $\ell\in BC^+(\Sigma)$. Then the following inequalities hold.
\begin{align}
\int_{\Sigma} \ell(x) \nu(dx)&=\int_{\Sigma} \ell(x) \left(\xi^{+}(dx)- \xi^{-}(dx)\right) + \int_{\Sigma} \ell(x) \mu(dx) \nonumber  \\
& \geq \inf_{x\in \Sigma} \ell(x){ \xi}^{+}(\Sigma) - \sup_{x\in \Sigma} \ell(x) { \xi}^-(\Sigma) + \int_{\Sigma} \ell(x) \mu(dx) \nonumber  \\
& = \inf_{x\in \Sigma} \ell(x)\frac{||{ \xi} ||_{TV}}{2} - \sup_{x\in \Sigma} \ell(x)\frac{||{ \xi} ||_{TV}}{2} + \int_{\Sigma} \ell(x) \mu(dx) \nonumber  \\
&=\left\{ \inf_{x\in \Sigma} \ell(x) - \sup_{x\in \Sigma} \ell(x) \right\}\frac{||{ \xi} ||_{TV}}{2} + \int_{\Sigma} \ell(x) \mu(dx). \label{fprblm2}
\end{align}

\noi The lower bound on the right hand side of (\ref{fprblm2}) is achieved by choosing ${ \xi}^* \in \widetilde{{\mathbb B }}_R({ \mu})$ as follows \bea
{ \xi}^*(dx) ={ \nu}^*(dx)-{ \mu}(dx) = \frac{R}{2} \left(\delta_{x_0}(dx) - \delta_{x^0}(dx)\right). \label{nmprblm2}
\eea

\noi This is a signed measure with total variation $||\xi^*||_{TV}=||{ \nu}^*-{ \mu}||_{TV}=R$. Hence, by using (\ref{nmprblm2}) as a candidate of the minimizing distribution then (\ref{fprblm2}) is equivalent to
\bea \int_{\Sigma} \ell(x) \nu^*(dx) =  \frac{R}{2}  \left\{ \inf_{x \in {\Sigma}} \ell(x)  -  \sup_{x \in \Sigma} \ell(x) \right\}    + E_{\mu}(\ell). \label{f2nproblm2}
\eea
\noi Solving the above equation with respect to $R$ the extremum Problem \ref{problem2} (for $D<E_{\mu}(\ell)$) is equivalent to
\vspace{-0.3cm}
\begin{align}
R^-(D)=\frac{2(D-E_{\mu}(\ell))}{\left\{\displaystyle \inf_{x \in {\Sigma}} \ell(x)  -  \sup_{x \in \Sigma}\ell(x)\right\}},
\end{align}

\noi where $\nu^*$ satisfies the constraint $\int_{\Sigma}\ell(x)\nu^*(dx)=D$, it is normalized $\nu^*(\Sigma)=1$, and $0\leq \nu(A)\leq 1$ on any $A \in {\cal B}(\Sigma)$. We can now identify $R_{\max}$ and $D_{\max}$ described in Lemma \ref{properties}. These are stated as a corollary.

\vspace{-0.2cm}
\begin{corollary}\label{cor.of.lemma}
\noi The values of $R_{\max}$ and $D_{\max}$ described in Lemma \ref{properties} are given by
\vspace{-0.1cm}
\bes
R_{\max}=2\left(1-\mu\left(\Sigma^0\right)\right) \hso\mbox{and}\hso D_{\max}=\int_{\Sigma}\ell(x)\mu(dx).
\ees
\end{corollary}

\begin{proof}
\noi Concerning $R_{\max}$, we know that $D^+(R)\leq \sup_{x\in \Sigma}\ell(x)$, $\forall R\geq0$, hence $D^+(R_{\max})$ can be at most $\sup_{x\in \Sigma}\ell(x)$. Since $D^+(R)$ is non-decreasing then $D^+(R_{\max})\leq D^+(R)\leq \sup_{x\in \Sigma}\ell(x)$, for any $R\geq R_{\max}$. Consider a $\nu$ that achieves this supremum. Let $\mu(\Sigma^0)$ and $\nu(\Sigma^0)$ to denote the nominal and true probability measures on $\Sigma^0$, respectively. If $\nu(\Sigma^0)=1$ then $\nu(\Sigma\setminus \Sigma^0)=0$. Therefore,
\begin{align*}
||\nu-\mu||_{TV}&=\sum_{x\in\Sigma^0}|\nu(x)-\mu(x)|+\sum_{x\in\Sigma\setminus\Sigma^0}|\nu(x)-\mu(x)|
\overset{(a)}=\sum_{x\in\Sigma^0}|\nu(x)-\mu(x)|+\sum_{x\in\Sigma\setminus\Sigma^0}|-\mu(x)|\\
&\overset{(b)}=\sum_{x\in\Sigma^0}\nu(x)-\sum_{x\in\Sigma^0}\mu(x)+\sum_{x\in\Sigma\setminus\Sigma^0}\mu(x)
=1-\sum_{x\in\Sigma^0}\mu(x)+\sum_{x\in\Sigma\setminus\Sigma^0}\mu(x)\\
&=2\left(1-\sum_{x\in\Sigma^0}\mu(x)\right)=2\left(1-\mu(\Sigma^0)\right),
\end{align*}

\noi where (a) follows due to $\nu(\Sigma\setminus\Sigma^0)=0$ which implies $\nu(x)=0$ for any $x\in\Sigma\setminus\Sigma^0$, and (b) follows because $\nu(x)\geq \mu(x)$ for all $x\in\Sigma^0$. Therefore, $R_{\max}=2(1-\mu(\Sigma^0))$ implies that $D^+(R_{\max})=\sup_{x\in\Sigma}\ell(x)$. Hence, $D^+(R)=\sup_{x\in \Sigma}\ell(x)$, for any $R\geq R_{\max}$.

Concerning $D_{\max}$, we know that $R^-(D)\geq 0$ for all $D\geq0$ hence $R^-(D_{\max})$ can be at least zero. Let $D_{\max}=\int_{\Sigma}\ell(x)\mu(dx)$, then it is obvious that $R^-(D_{\max})=0$. Since $R^-(D)$ in non-increasing, then $0\leq R^-(D)\leq R^-(D_{\max})$, for any $D\geq D_{\max}$. Hence, $ R^-(D)=0$, for any $D\geq D_{\max}$.
\end{proof}

%
%
%
%
\section{Characterization of Extremum Measures for Finite Alphabets} \label{cem}

This section uses the results of Section~\ref{char.extr.meas.} to compute closed form expressions for the extremum measures $\nu^*$ for any $R \in [0,2]$, when $\Sigma$ is a finite alphabet space to give the intuition into the solution procedure. This is done by identifying the sets $\Sigma^0$, $\Sigma_0$, $\Sigma\setminus\Sigma^0 \cup \Sigma_0$, and the measure $\nu^*$ on these sets for any $R\in [0,2]$. Although this can be done for probability measures on complete separable metric spaces (Polish spaces) $(\Sigma, d_\Sigma)$, and for $\ell\in BM^+(\Sigma)$, $\ell\in BC^+(\Sigma)$, $L^{\infty,+}(\Sigma,{\cal B}(\Sigma),\nu)$, we prefer to discuss the finite alphabet case to gain additional insight into these problems. At the end of this section we shall use the finite alphabet case to discuss the extensions to countable alphabet and to $\ell\in L^{\infty,+}(\Sigma,{\cal B}(\Sigma),\nu)$.

Consider the finite alphabet case $(\Sigma,{\cal M})$, where $\mbox{card}(\Sigma)=|\Sigma|$ is finite, ${\cal M}=2^{|\Sigma|}$. Thus, $\nu$ and $\mu$ are point mass distributions on $\Sigma$. Define the set of probability vectors on $\Sigma$ by
\vspace{-0.1cm}
\begin{align}
{\mathbb P}({\Sigma}) \tri \Big\{p =( p_1, \ldots, p_{|{\Sigma}|}):p_i\geq 0,i=0,\ldots,|\Sigma|, \sum_{i \in {\Sigma}} p_i =1 \Big\}.
\end{align}
 \noi Thus, $p\in {\mathbb P}({\Sigma})$ is a probability vector in ${\mathbb R}_+^{|\Sigma|}$. Also let $\ell\tri\{\ell_1,\ldots,\ell_{|\Sigma|}\}$ so that $\ell\in {\mathbb R}_+^{|\Sigma|}$ (e.g., set of non-negative vectors of dimension $|\Sigma|$).

\subsection{Problem \ref{problem1}: Finite Alphabet Case}
\label{subsec.problem1}
\noi Suppose $\nu \in {\mathbb P}({\Sigma})$ is the true probability vector and $\mu\in{\mathbb P}({\Sigma})$ is the nominal fixed probability vector. The extremum problem is defined by \bea  \label{prob.finite.case}D^+(R) \tri    \max_{ { \nu} \in  {\mathbb B}_R({ \mu})}  \sum_{i\in \Sigma} \ell_i\nu_i,
\eea
\vspace{-0.1cm}
\noi where
\vspace{-0.3cm}
\bea
{\mathbb B }_R({ \mu})   \tri \Big\{ {\bf \nu} \in {\mathbb P}(\Sigma)): ||{ \nu} -{ \mu}||_{TV}\triangleq\sum_{i\in\Sigma}|\nu_i-\mu_i| \leq R \Big\}. \label{b323}
\eea
\vspace{-0.1cm}
Next, we apply the results of Section \ref{char.extr.meas.} to characterize the optimal $\nu^*$ for any $R\in[0,2]$. By defining, $\xi_i \tri \nu_i-\mu_i, i=1, \ldots,|\Sigma|$ and $\xi \in {\mathbb M}_{0}(\Sigma)$,  Problem \ref{problem1} can be reformulated as follows. \bea  \label{reform.equation}\max_{\nu \in{\mathbb B }_R({ \mu})}\sum_{i\in\Sigma}\ell_i\nu_i\longrightarrow\sum_{i\in\Sigma}\ell_i \mu_i+\max_{\xi \in\widetilde{{\mathbb B }}_R({ \mu})}\sum_{i\in\Sigma} \ell_i \xi_i.\eea
\noi Note that $\xi \in\widetilde{{\mathbb B }}_R({ \mu})$ is described by the  constraints \bea \sum_{i \in \Sigma} |\xi_i|\leq R, \hso \sum_{i \in \Sigma} \xi_i=0,\hso 0\leq \xi_i + \mu_i\leq 1, \hso\forall i \in \Sigma. \eea

\noi The positive and negative variation of the signed measure $\xi$ are defined by
\[ {\xi_i^+\triangleq}\begin{cases}
\xi_i, & \mbox{if}~ \xi_i \geq 0 \\
0, & \mbox{if}~ \xi_i < 0,
\end{cases}
\qquad
{\xi_i^-\triangleq}\begin{cases}
0, & \mbox{if}~ \xi_i \geq 0 \\
-\xi_i, & \mbox{if}~ \xi_i < 0,
\end{cases}
\]

\noi Therefore,
\vspace{-0.3cm}
\bes
\sum_{i\in\Sigma}\xi_i=\sum_{i\in\Sigma}\xi_i^+-\sum_{i\in\Sigma}\xi_i^-,\hst\sum_{i\in\Sigma}|\xi_i|=\sum_{i\in\Sigma}\xi_i^++\sum_{i\in\Sigma}\xi_i^-, 
\ees

\noi and hence,
\bes \sum_{i\in\Sigma}\xi_i^+=\frac{\sum_{i\in\Sigma}\xi_i+\sum_{i\in\Sigma}|\xi_i|}{2}\equiv\frac{\alpha}{2},\hst\sum_{i\in\Sigma}\xi_i^-=\frac{-\sum_{i\in\Sigma}\xi_i+\sum_{i\in\Sigma}|\xi_i|}{2}\equiv\frac{\alpha}{2},
\ees

\noi and
\vspace{-0.3cm}
\bea
\sum_{i\in\Sigma} \xi_i=0, \hst\alpha=\sum_{i\in\Sigma}|\xi_i|\leq R.\eea
\noi In addition,
\vspace{-0.3cm}
\bea
\sum_{i\in\Sigma} \ell_i\xi_i=\sum_{i\in\Sigma} \ell_i\xi_i^+-\sum_{i\in\Sigma} \ell_i\xi_i^-.\eea

\noi Define the maximum and minimum values of the sequence $\{\ell_1,\hdots,\ell_{|\Sigma|}\}$ by $\ell_{\max}\triangleq \max_{i\in\Sigma}\ell_i$, $\ell_{\min}\triangleq \min_{i\in\Sigma}\ell_i$, and its corresponding support sets by
$\Sigma^0 \triangleq \{i\in \Sigma:\ell_i=\ell_{\max} \}$, $\Sigma_0\triangleq\{i\in \Sigma:\ell_i=\ell_{\min} \}$.
\noi For all remaining sequence, $\{\ell_i:i\in \Sigma \setminus \Sigma^0\cup\Sigma_0\}$, and for $1\leq r \leq |\Sigma\setminus\Sigma^0\cup\Sigma_0|$ define recursively
\vspace{-0.1cm}
\begin{align}\label{sigmasets1}
\Sigma_k \triangleq \left\{i\in \Sigma:\ell_i=\min\left\{\ell_{\alpha}: \alpha \in \Sigma\setminus \Sigma^0\cup\left(\bigcup_{j=1}^k\Sigma_{j-1}\right)\right\} \right\},\hst  k\in\{1,2,\hdots, r\},\end{align}
\noi till all the elements of $\Sigma$ are exhausted (i.e., $k$ is at most $|\Sigma\setminus\Sigma^0\cup\Sigma_0|$). Define the corresponding values of the sequence of sets in (\ref{sigmasets1}) by
 \bes \ell(\Sigma_k)\triangleq\min_{i\in\Sigma\setminus\Sigma^0\cup(\bigcup_{j=1}^k\Sigma_{j-1})}\ell_i,\hst k\in\{1,2,\hdots, r\},\ees
\noi where $r$ is the number of $\Sigma_k$ sets which is at most $|\Sigma\setminus\Sigma^0\cup\Sigma_0|$; for example, when $k=1$, $\ell\left(\Sigma_1\right)=\min_{i\in\Sigma\setminus\Sigma^0\cup\Sigma_0}\ell_i$. The following theorem characterizes the solution of Problem \ref{problem1}.

\begin{theorem}\label{thmproblm3.1}
\noi The solution of the finite alphabet version of Problem \ref{problem1} is given by
\begin{align} D^+(R)=\ell_{\max}\nu^*(\Sigma^0)+\ell_{\min}\nu^*(\Sigma_0)+\sum_{k=1}^r\ell(\Sigma_k)\nu^*(\Sigma_k).\label{mp}\end{align}
\noi Moreover, the optimal probabilities are given by
\begin{subequations}
\label{all3}
\begin{align}
& \nu^*(\Sigma^0)\triangleq\sum_{i\in\Sigma^0}\nu_i^*=\sum_{i\in\Sigma^0}\mu_i+\alpha,\label{all3a}\\
& \nu^*(\Sigma_0)\triangleq\sum_{i\in\Sigma_0}\nu_i^*=\left(\sum_{i\in\Sigma_0}\mu_i-\alpha\right)^+,\label{all3b}\\
& \nu^*(\Sigma_k)\triangleq\sum_{i\in\Sigma_k} \nu_i^*=\left(\sum_{i\in\Sigma_k} \mu_i-\left(\alpha-\sum_{j=1}^k\sum_{i\in\Sigma_{j-1}}\mu_i\right)^+\right)^+, \label{all3c}\\
&\alpha=\min\left(\frac{R}{2},1-\sum_{i\in\Sigma^0}\mu_i\right),\label{all3d}
\end{align}
\end{subequations}

\noi where, $k=1,2,\hdots,r$ and $r$ is the number of $\Sigma_k$ sets which is at most $|\Sigma\setminus\Sigma^0\cup\Sigma_0|$.
\end{theorem}

\begin{proof}
\noi The derivation of the Theorem is based on a sequence of Lemmas, Propositions and Corollaries which are presented below.
\end{proof}

\noi The following Lemma is a direct consequence of Section \ref{subsec.Equiv.Extr.Prob}.

\begin{lemma}\label{lemmabound}
\noi Consider the finite alphabet version of Problem \ref{problem1}. Then the following bounds hold.\\
\noi 1. Upper Bound.
\vspace{-0.3cm}
\bea
\sum_{i\in\Sigma}\ell_i\xi_i^+\leq\ell_{\max}\left(\frac{\alpha}{2}\right). \eea
\noi The upper bound holds with equality if \bea\sum_{i\in\Sigma^0} \mu_i+\frac{\alpha}{2}\leq 1,\hso\sum_{i\in\Sigma^0}\xi_i^+=\frac{\alpha}{2},\hso\xi_i^+=0\hso \mbox{for}\hso i\in\Sigma\setminus\Sigma^0,\eea
\noi and the optimal probability on $\Sigma^0$ is given by
\bea\label{opt.prob.on.S^0}\nu^*(\Sigma^0)\triangleq \sum_{i\in\Sigma^0}\nu_i^*=\min\left(1,\sum_{i\in\Sigma^0}\mu_i+\frac{\alpha}{2}\right).\eea

\noi 2. Lower Bound.
\vspace{-0.3cm}
\bea \sum_{i\in\Sigma}\ell_i\xi_i^-\geq\ell_{\min}\left(\frac{\alpha}{2}\right). \eea
\noi The lower bound holds with equality if \bea\sum_{i\in \Sigma_0}\mu_i-\frac{\alpha}{2}\geq 0,\hso \sum_{i\in \Sigma_0}\xi_i^-=  \frac{\alpha}{2},\hso \xi_i^-=0\hso\mbox{for}\hso i\in \Sigma\setminus\Sigma_0,\eea
\noi and the optimal probability on $\Sigma_0$ is given by
\bea\label{opt.prob.on.S_0}\nu^*(\Sigma_0)\triangleq \sum_{i\in\Sigma_0}\nu_i^*=\left(\sum_{i\in\Sigma_0}\mu_i-\frac{\alpha}{2}\right)^+.\eea\\
\noi Moreover, under the conditions in 1 and 2 the maximum pay-off is given by
\begin{align}\label{opt.pay.off1} D^+(R)= \frac{\alpha}{2}\{\ell_{\max}-\ell_{\min} \}+\sum_{i\in\Sigma}\ell_i\mu_i.\end{align}
\end{lemma}

\begin{proof}
\noi Follows from Section \ref{subsec.Equiv.Extr.Prob}.
\end{proof}

\begin{proposition}\label{subcase1}
\noi If $\sum_{i\in\Sigma^0}\mu_i+\frac{\alpha}{2}=1$ then $D^+(R)=\ell_{\max}$.
\end{proposition}

\begin{proof}
\noi Under the stated condition $\sum_{i\in\Sigma^0}\nu^*_i=1$  and therefore $\sum_{i\in\Sigma\setminus\Sigma^0}\nu^*_i=0 $, hence $\nu^*_i=0,$ for all $i\in\Sigma\setminus\Sigma^0 $. Then the maximum pay-off (\ref{prob.finite.case}) is given by
\begin{align}
D^+(R)&=\sum_{i\in\Sigma^0}\ell_i\nu^*_i+\sum_{i\in\Sigma\setminus\Sigma^0}\ell_i\nu^*_i=\ell_{\max}\sum_{i\in\Sigma^0}\nu^*_i=\ell_{\max}. \nonumber
\end{align}
\end{proof}
%
%
%
\vspace{-0.3cm}
\noi The lower bound of Lemma \ref{lemmabound} characterize the extremum solution for $\sum_{i\in\Sigma_0}\mu_i-\frac{\alpha}{2}\geq 0$. Next, the characterization of extremum solution is discussed when this condition is violated.

\begin{lemma}\label{subcase2}
\noi If $\sum_{i\in\Sigma_0}\mu_i-\frac{\alpha}{2}\leq 0$, then \bea\label{eq1.subcase.2}\sum_{i\in\Sigma}\ell_i\xi_i^-\geq \ell(\Sigma_1)\left(\frac{\alpha}{2}-\sum_{i\in\Sigma_0}\mu_i\right)+\ell_{\min}\sum_{i\in\Sigma_0}\mu_i.\eea

\noi Moreover, equality holds if
\vspace{-0.3cm}
\begin{subequations}
\label{all1}
\begin{align}
&\sum_{i\in\Sigma_0}\xi_i^-=\sum_{i\in\Sigma_0}\mu_i,\label{all1a}\\
&\sum_{i\in \Sigma_1}\xi_i^-=\left(\frac{\alpha}{2}-\sum_{i\in\Sigma_0}\mu_i\right),\label{all1b}\\
&\sum_{i\in\Sigma_0}\mu_i+\sum_{i\in\Sigma_1}\mu_1\geq \frac{\alpha}{2},\label{all1c}\\
&\xi_i^-=0 \hso\mbox{for all}\hso i\in\Sigma\setminus\Sigma_0\cup\Sigma_1,
\end{align}
\end{subequations}
\noi and the optimal probability on $\Sigma_1$ is given by \bea \label{opt.prob.sigma1} \sum_{i\in\Sigma_1}\nu^*_i=\left(\sum_{i\in\Sigma_1}\mu_i-\left(\frac{\alpha}{2}-\sum_{i\in\Sigma_0}\mu_i\right)^+\right)^+. \eea

\end{lemma}

\begin{proof}
\noi First, we show that inequality holds.
\begin{align*}
\sum_{i\in\Sigma\setminus\Sigma_0}\ell_i\xi_i^-\geq\min_{i\in\Sigma\setminus\Sigma_0}\ell_i\sum_{i\in\Sigma\setminus\Sigma_0}\xi_i^-=\ell(\Sigma_1)\sum_{i\in\Sigma\setminus\Sigma_0}\xi_i^-=\ell(\Sigma_1)\left(\sum_{i\in\Sigma}\xi_i^--\sum_{i\in\Sigma_0}\xi_i^-\right).
\end{align*}
\noi Hence,
\vspace{-0.1cm}
\begin{align*}
&\sum_{i\in\Sigma}\ell_i\xi_i^--\sum_{i\in\Sigma_0}\ell_i\xi_i^-\geq\ell(\Sigma_1)\left(\frac{\alpha}{2}-\sum_{i\in\Sigma_0}\mu_i\right),
\end{align*}
\noi which implies
\begin{align*}
&\sum_{i\in\Sigma}\ell_i\xi_i^-\geq\ell(\Sigma_1)\left(\frac{\alpha}{2}-\sum_{i\in \Sigma_0}\mu_i\right)+\ell_{\min}\sum_{i\in\Sigma_0}\mu_i,
\end{align*}

\noi establishing (\ref{eq1.subcase.2}). Next, we show under the stated conditions that equality holds.
\begin{align} \sum_{i\in\Sigma}\ell_i\xi_i^-&=\sum_{i\in\Sigma_0}\ell_i\xi_i^-+\sum_{i\in\Sigma_1}\ell_i\xi_i^-+\sum_{i\in\Sigma\setminus\Sigma_0\cup\Sigma_1}\ell_i\xi_i^-\nonumber\\
&=\ell_{\min}\sum_{i\in\Sigma_0}\mu_i+\ell(\Sigma_1)\sum_{i\in\Sigma_1}\xi_i^-=\ell_{\min}\sum_{i\in\Sigma_0}\mu_i+\ell(\Sigma_1)\left(\frac{\alpha}{2}-\sum_{i\in\Sigma_0}\mu_i\right).\nonumber
\end{align}

\noi From (\ref{all1b}) we have that
\bea \sum_{i\in\Sigma_1}\xi_i^-=\left(\frac{\alpha}{2}-\sum_{i\in\Sigma_0}\mu_i\right),\eea
\noi and hence,
\bea \sum_{i\in\Sigma_1}\nu_i=\sum_{i\in\Sigma_1}\mu_i-\left(\frac{\alpha}{2}-\sum_{i\in\Sigma_0}\mu_i\right).\eea

\noi The optimal $\sum_{i\in\Sigma_1}\nu_i$ must satisfy $\frac{\alpha}{2}-\sum_{i\in\Sigma_0}\mu_i\geq 0$ and $\sum_{i\in\Sigma_1}\mu_i+\sum_{i\in\Sigma_0}\mu_i-\frac{\alpha}{2}\geq0$. Hence, (\ref{opt.prob.sigma1}) is obtained.
\end{proof}

\noi Following the previous Lemma, which characterizes the extremum solution when $\sum_{i\in\Sigma_0}\mu_i-\frac{\alpha}{2}\leq 0$, one can also characterize the optimum solution of extremum Problem \ref{problem1}, when $\sum_{j=1}^k\sum_{i\in\Sigma_{j-1}}\mu_i-\frac{\alpha}{2}\leq 0$, for any $k\in\{1,2,\hdots,r\}$.

\begin{corollary}\label{subcase3}
\noi For any $k\in\{1,2,\hdots,r\}$, if $\sum_{j=1}^k\sum_{i\in\Sigma_{j-1}}\mu_i-\frac{\alpha}{2}\leq 0$ then
\bea\sum_{i\in\Sigma}\ell_i\xi_i^-\geq \ell(\Sigma_k)\left(\frac{\alpha}{2}-\sum_{j=1}^k\sum_{i\in\Sigma_{j-1}}\mu_i\right)+\sum_{j=1}^k\sum_{i\in\Sigma_{j-1}}\ell_i\mu_i.\eea

\noi Moreover, equality holds if
\begin{subequations}
\label{all2}
\begin{align}
&\sum_{i\in\Sigma_{j-1}}\xi_i^-=\sum_{i\in\Sigma_{j-1}}\mu_i, \hso\mbox{for all}\hso j=1,2,\hdots,k,\label{all2a}\\
&\sum_{i\in \Sigma_k}\xi_i^-=\left(\frac{\alpha}{2}-\sum_{j=1}^k\sum_{i\in\Sigma_{j-1}}\mu_i\right),\label{all2b}\\
&\sum_{j=0}^k\sum_{i\in\Sigma_{j}}\mu_i\geq \frac{\alpha}{2},\label{all2c}\\
&\xi_i^-=0 \hso\mbox{for all}\hso i\in\Sigma\setminus\Sigma_0\cup\Sigma_1\cup\hdots\cup\Sigma_k,
\end{align}
\end{subequations}
\noi and the optimal probability on $\Sigma_k$ sets is given by \bea \label{opt.prob.sigmak} \sum_{i\in\Sigma_k}\nu^*_i=\left(\sum_{i\in\Sigma_k}\mu_i-\left(\frac{\alpha}{2}-\sum_{j=1}^k\sum_{i\in\Sigma_{j-1}}\mu_i\right)^+\right)^+. \eea

\end{corollary}

\begin{proof}
\noi Consider any $k\in\{1,2,\ldots,r\}$. First, we show that inequality holds. From lower bound we have that
\begin{align*}
\sum_{i\in\Sigma\setminus\displaystyle\cup_{j=1}^k\Sigma_{j-1}}\ell_i\xi_i^-&\geq\min_{i\in\Sigma\setminus\displaystyle\cup_{j=1}^k\Sigma_{j-1}}\ell_i\sum_{i\in\Sigma\setminus\displaystyle\cup_{j=1}^k\Sigma_{j-1}}\xi_i^-\\
&=\ell(\Sigma_k)\sum_{i\in\Sigma\setminus\displaystyle\cup_{j=1}^k\Sigma_{j-1}}\xi_i^-=\ell(\Sigma_k)\left(\sum_{i\in\Sigma}\xi_i^--\sum_{j=1}^k\sum_{i\in\Sigma_{j-1}}\xi_i^-\right).
\end{align*}

\noi Hence, \begin{align*}\sum_{i\in\Sigma}\ell_i\xi_i^--\sum_{j=1}^k\sum_{i\in\Sigma_{j-1}}\ell_i\xi_i^-\geq \ell(\Sigma_k)\left(\frac{\alpha}{2}-\sum_{j=1}^k\sum_{i\in\Sigma_{j-1}}\mu_i\right),\end{align*}

\noi which implies  \begin{align*}
\sum_{i\in\Sigma}\ell_i\xi_i^-\geq\ell(\Sigma_k)\left(\frac{\alpha}{2}-\sum_{j=1}^k\sum_{i\in\Sigma_{j-1}}\mu_i\right)+\sum_{j=1}^k\sum_{i\in\Sigma_{j-1}}\ell_i\mu_i.
\end{align*}

\noi Next, we show under the stated conditions that equality holds.
\begin{align*} \sum_{i\in\Sigma}\ell_i\xi_i^-&=\sum_{j=1}^k\sum_{i\in\Sigma_{j-1}}\ell_i\xi_i^-+\sum_{i\in\Sigma_k}\ell_i\xi_i^-+\sum_{i\in\Sigma\setminus\displaystyle\cup_{j=0}^k\Sigma_{j}}\ell_i\xi_i^-\\
&=\sum_{j=1}^k\ell(\Sigma_{j-1})\sum_{i\in\Sigma_{j-1}}\xi_i^-+\ell(\Sigma_k)\sum_{i\in\Sigma_k}\xi_i^-
=\sum_{j=1}^k\sum_{i\in\Sigma_{j-1}}\ell_i\mu_i+\ell(\Sigma_k)\left(\frac{\alpha}{2}-\sum_{j=1}^k\sum_{i\in\Sigma_{j-1}}\mu_i\right).
\end{align*}

\noi From (\ref{all2b}) we have that
\vspace{-0.0cm}
\bea \sum_{i\in \Sigma_k}\xi_i^-=\left(\frac{\alpha}{2}-\sum_{j=1}^k\sum_{i\in\Sigma_{j-1}}\mu_i\right),\eea
\noi and hence,
\vspace{-0.3cm}
\bea \sum_{i\in\Sigma_k}\nu_i=\sum_{i\in\Sigma_k}\mu_i-\left(\frac{\alpha}{2}-\sum_{j=1}^k\sum_{i\in\Sigma_{j-1}}\mu_i\right).\eea

\noi The optimal $\sum_{i\in\Sigma_k}\nu_i^*$ must satisfy $\frac{\alpha}{2}-\sum_{j=1}^k\sum_{i\in\Sigma_{j-1}}\mu_i\geq 0$ and $\sum_{j=0}^k\sum_{i\in\Sigma_j}\mu_i-\frac{\alpha}{2}\geq0$. Hence, (\ref{opt.prob.sigmak}) is obtained.
\end{proof}

\noi Putting together Lemma \ref{lemmabound}, Proposition \ref{subcase1}, Lemma \ref{subcase2}, and Corollary \ref{subcase3} we obtain the result of Theorem \ref{thmproblm3.1}. Notice that the solution of Problem \ref{problem1} finds the partition of $\Sigma$ into disjoint sets $\{\Sigma^0,\Sigma_0,\Sigma_1,\ldots,\Sigma_k\}$, where $\Sigma=\Sigma^0\cup\Sigma_0\cup\Sigma_1\cup\ldots\cup\Sigma_k$, and the optimal measure $\nu^*(\cdot)$ on these sets.

\subsection{Problem \ref{problem2}: Finite Alphabet Case}
\noi Consider Problem \ref{problem2}, and follow the procedure utilized to derive the solution of Problem \ref{problem1} (e.g., Section \ref{subsec.problem1}).
\noi Let $\xi_i\triangleq \nu_i-\mu_i\equiv \xi_i^+-\xi_i^-$, be the signed measure decomposition of $\xi$. We know that, $\sum_{i\in\Sigma}\xi_i=0$ and so, $\sum_{i\in\Sigma}\xi_i^+=\sum_{i\in\Sigma}\xi_i^-$. Also
\bea \sum_{i\in\Sigma}|\nu_i-\mu_i|=\sum_{i\in\Sigma}|\xi_i|=\sum_{i\in\Sigma}\xi_i^++\sum_{i\in\Sigma}\xi_i^-=\alpha,\hso \sum_{i\in\Sigma}\xi_i^+=\sum_{i\in\Sigma}\xi_i^-=\frac{\alpha}{2}.\eea

\noi The average constraint can be written as follows
\begin{align}\label{alt.dual.min.average.constraint}
\sum_{i\in\Sigma}\ell_i\nu_i&=\sum_{i\in\Sigma}\ell_i(\xi_i+\mu_i)=\sum_{i\in\Sigma}\ell_i\xi_i+\sum_{i\in\Sigma}\ell_i\mu_i
=\sum_{i\in\Sigma}\ell_i\xi_i^+-\sum_{i\in\Sigma}\ell_i\xi_i^-+\sum_{i\in\Sigma}\ell_i\mu_i\leq D.
\end{align}

\noi Define the maximum and minimum values of the sequence by $\ell_{\max}\triangleq \max_{i\in\Sigma}\ell_i$, $\ell_{\min}\triangleq \min_{i\in\Sigma}\ell_i$ and its corresponding support sets by
$\Sigma^0 \triangleq \{i\in \Sigma:\ell_i=\ell_{\max} \}$, $\Sigma_0\triangleq\{i\in \Sigma:\ell_i=\ell_{\min} \}$.
\noi For all remaining sequence, $\{\ell_i:i\in \Sigma \setminus \Sigma^0\cup\Sigma_0\}$, and for $1\leq r \leq |\Sigma\setminus\Sigma^0\cup\Sigma_0|$ define recursively
\begin{align} \label{sigmasetvalue}\Sigma^k \triangleq \left\{i\in \Sigma:\ell_i=\max\left\{\ell_{\alpha}: \alpha \in \Sigma\setminus \Sigma_0\cup\left(\bigcup_{j=1}^k\Sigma^{j-1}\right)\right\} \right\},\hst  k\in\{1,2,\hdots, r\},\end{align}
\noi till all the elements of $\Sigma$ are exhausted, and define the corresponding maximum value of $\ell$ on the sequence on these sets by
 \bes \ell\left(\Sigma^k\right)\triangleq\max_{i\in\Sigma\setminus\Sigma_0\cup\left(\bigcup_{j=1}^k\Sigma^{j-1}\right)}\ell_i,\hst k\in\{1,2,\hdots, r\},\ees
\noi where $r$ is the number of $\Sigma^k$ sets which is at most $|\Sigma\setminus\Sigma^0\cup\Sigma_0|$. Clearly, $\ell\left(\Sigma^1\right)=\max_{i\in\Sigma\setminus\Sigma^0\cup\Sigma_0}\ell_i$ and so on. Note the analogy between (\ref{sigmasetvalue}) and (\ref{sigmasets1}) for Problem \ref{problem1}. The main theorem which characterizes the extremum solution of Problem \ref{problem2} is given below.

\begin{theorem}\label{thmproblm3.2}
\noi The solution of the finite alphabet version of Problem \ref{problem2} is given by
\begin{align}\label{dualmp} R^-(D)=\sum_{i\in\Sigma}|\nu_i^*-\mu_i|, \end{align}

\noi where the value of $R^-(D)$ is calculated as follows.
\bi
\item [(1)] If
\begin{align*}\displaystyle \ell_{\min}\left(\sum_{j=0}^k\sum_{i\in\Sigma^j}\mu_i+\sum_{i\in\Sigma_0}\mu_i\right)+\sum_{j=k+1}^r\sum_{i\in\Sigma^j}\ell_i\mu_i\leq D\leq \ell_{\min}\left(\sum_{j=1}^k\sum_{i\in\Sigma^{j-1}}\mu_i+\sum_{i\in\Sigma_0}\mu_i\right)+\sum_{j=k}^r\sum_{i\in\Sigma^j}\ell_i\mu_i
    \end{align*} then
\bea R^-(D)=\frac{\displaystyle 2\left(D-\ell_{\min}\sum_{i\in\Sigma_0}\mu_i-\ell\left(\Sigma^k\right)\sum_{j=1}^{k}\sum_{i\in\Sigma^{j-1}}\mu_i-\sum_{j=k}^r\sum_{i\in\Sigma^j}\ell_i\mu_i\right)}{\ell_{\min}-\ell\left(\Sigma^k\right)}.\label{eq5.2}\eea
\item [(2)] If $D\geq \left(\ell_{\min}-\ell_{\max}\right)\sum_{i\in\Sigma^0}\mu_i+\sum_{i\in\Sigma}\ell_i\mu_i$ then \bea R^-(D)=\frac{\displaystyle 2\left(D-\sum_{i\in\Sigma}\ell_i\mu_i\right)}{\ell_{\min}-\ell_{\max}}.\label{eq5.1}\eea
\ei
\noi Moreover, the optimal probabilities are given by
\begin{subequations}
\begin{align}
&\nu^*(\Sigma_0) \triangleq\sum_{i\in\Sigma_0}\nu_i^*=\sum_{i\in\Sigma_0}\mu_i+\alpha,\\
&\nu^*(\Sigma^0) \triangleq\sum_{i\in\Sigma^0}\nu_i^*=\left(\sum_{i\in\Sigma^0}\mu_i-\alpha\right)^+,\\
&\nu^*(\Sigma^k) \triangleq\sum_{i\in\Sigma^k} \nu_i^*=\left(\sum_{i\in\Sigma^k} \mu_i-\left(\alpha-\sum_{j=1}^k\sum_{i\in\Sigma^{j-1}}\mu_i\right)^+\right)^+, \\
&\alpha=\min\left(\frac{R^-(D)}{2},1-\sum_{i\in\Sigma_0}\mu_i\right).
\end{align}
\end{subequations}
where $k=1,2,\hdots,r$ and $r$ is the number of $\Sigma^k$ sets which is at most $|\Sigma\setminus\Sigma^0\cup\Sigma_0|$.
\end{theorem}
\begin{proof}
\noi For the derivation of the Theorem see Appendix \ref{app1}.\end{proof}

\subsection{Solutions of Related Extremum Problems}
\label{subsec.rel.extr.prblm}
In Section \ref{subsec.Rel.Extr.Prob} we discuss related extremum problems, whose solution can be obtained from those of Problem \ref{problem1} and Problem \ref{problem2}. In this Section we give the solution of the finite alphabet version of the related extremum problems described by (\ref{primal1}) and (\ref{rep4}).

Consider the finite alphabet version of (\ref{primal1}), that is
\vspace{-0.1cm}
 \bea \label{dual11}R^+(D)\tri\sup_{\nu\in {\cal M}_1(\Sigma):\sum_{i\in\Sigma}\ell_i\nu_i\leq D}||\nu-\mu||_{TV}.\eea

\noi The solution of (\ref{dual11}) is obtained from the solution of Problem \ref{problem1}, by finding the inverse mapping or by following a similar procedure to the one utilized to derive Theorem \ref{thmproblm3.2}. 

\vspace{-0.1cm}
\begin{theorem}\label{thmdualproblm3.1}
\noi The solution of the finite alphabet version of (\ref{dual11}) is given by
\vspace{-0.2cm}
\begin{align} R^+(D)=\sum_{i\in\Sigma}|\nu_i^*-\mu_i|, \label{dualmp11}\end{align}
\noi where the value of $R^+(D)$ is calculated as follows.
\bi
\item [(1)]If
\vspace{-0.3cm}
\begin{align*}\displaystyle\ell_{\max}\left(\sum_{j=1}^k\sum_{i\in\Sigma_{j-1}}\mu_i+\sum_{i\in\Sigma^0}\mu_i\right)+\sum_{j=k}^r\sum_{i\in\Sigma_j}\ell_i\mu_i\leq D\leq\ell_{\max}\left(\sum_{j=0}^k\sum_{i\in\Sigma_{j}}\mu_i+\sum_{i\in\Sigma^0}\mu_i\right)+\sum_{j=k+1}^r\sum_{i\in\Sigma_j}\ell_i\mu_i\end{align*} then
\vspace{-0.3cm}
\bea R^+(D)=\frac{\displaystyle 2\left(D-\ell_{\max}\sum_{i\in\Sigma^0}\mu_i-\ell\left(\Sigma_k\right)\sum_{j=1}^{k}\sum_{i\in\Sigma_{j-1}}\mu_i-\sum_{j=k}^r\sum_{i\in\Sigma_j}\ell_i\mu_i\right)}{\ell_{\max}-\ell\left(\Sigma_k\right)}.\eea
\item [(2)] If $\displaystyle D\leq \left(\ell_{\max}-\ell_{\min}\right)\sum_{i\in\Sigma_0}\mu_i+\sum_{i\in\Sigma}\ell_i\mu_i$ then
\vspace{-0.3cm}
    \bea R^+(D)=\frac{\displaystyle 2\left(D-\sum_{i\in\Sigma}\ell_i\mu_i\right)}{\ell_{\max}-\ell_{\min}}.\eea
\ei
\noi Moreover, the optimal probabilities are given by
\begin{subequations}
\begin{align}
& \nu^*(\Sigma^0)\triangleq\sum_{i\in\Sigma^0}\nu_i^*=\sum_{i\in\Sigma^0}\mu_i+\alpha,\label{all3a}\\
& \nu^*(\Sigma_0)\triangleq\sum_{i\in\Sigma_0}\nu_i^*=\left(\sum_{i\in\Sigma_0}\mu_i-\alpha\right)^+,\label{all3b}\\
& \nu^*(\Sigma_k)\triangleq\sum_{i\in\Sigma_k} \nu_i^*=\left(\sum_{i\in\Sigma_k} \mu_i-\left(\alpha-\sum_{j=1}^k\sum_{i\in\Sigma_{j-1}}\mu_i\right)^+\right)^+, \label{all3c}\\
&\alpha=\min\left(\frac{R^+(D)}{2},1-\sum_{i\in\Sigma^0}\mu_i\right).\label{all3d}
\end{align}
\end{subequations}
\noi where, $k=1,2,\hdots,r$ and $r$ is the number of $\Sigma_k$ sets which is at most $|\Sigma\setminus\Sigma^0\cup\Sigma_0|$.
\end{theorem}

\noi Consider the finite alphabet version of (\ref{rep4}), that is
\bea D^-(R)\tri\inf_{\nu\in{\cal M}_1(\Sigma):||\nu-\mu||_{TV}\leq R}\sum_{i\in\Sigma}\ell_i\nu_i.\label{dual12}\eea

\noi The solution of (\ref{dual12}) is obtained from that of Problem \ref{problem1}, but with a reverse computation on the partition of $\Sigma$ and the mass of the extremum measure on the partition moving in the opposite direction. Below, we give the main theorem.

\begin{theorem}\label{thmdualproblm3.2}
\noi The solution of the finite alphabet version of (\ref{dual12}) is given by
\begin{align} D^-(R)=\ell_{\max}\nu^*(\Sigma^0)+\ell_{\min}\nu^*(\Sigma_0)+\sum_{k=1}^r\ell(\Sigma^k)\nu^*(\Sigma^k).\label{dualmp30}\end{align}
\noi Moreover, the optimal probabilities are given by
\begin{subequations}
\begin{align}
& \nu^*(\Sigma_0)\triangleq\sum_{i\in\Sigma_0}\nu_i^*=\sum_{i\in\Sigma_0}\mu_i+\alpha,\\
& \nu^*(\Sigma^0)\triangleq\sum_{i\in\Sigma^0}\nu_i^*=\left(\sum_{i\in\Sigma^0}\mu_i-\alpha\right)^+,\\
& \nu^*(\Sigma^k)\triangleq\sum_{i\in\Sigma^k} \nu_i^*=\left(\sum_{i\in\Sigma^k} \mu_i-\left(\alpha-\sum_{j=1}^k\sum_{i\in\Sigma^{j-1}}\mu_i\right)^+\right)^+, \\
&\alpha=\min\left(\frac{R}{2},1-\sum_{i\in\Sigma_0}\mu_i\right),
\end{align}
\end{subequations}

\noi where, $k=1,2,\hdots,r$ and $r$ is the number of $\Sigma^k$ sets which is at most $|\Sigma\setminus\Sigma^0\cup\Sigma_0|$.
\end{theorem}

\begin{remark}
\noi The statements of Theorems \ref{thmproblm3.1}, \ref{thmproblm3.2}, \ref{thmdualproblm3.1}, \ref{thmdualproblm3.2} are also valid for the countable alphabet case, because their derivations are not restricted to $\Sigma$ being finite alphabet. It also holds for any $\ell\in BC^+(\Sigma)$ as seen in Section \ref{char.extr.meas.}. The extensions of Theorems \ref{thmproblm3.1}-\ref{thmdualproblm3.2} to $\ell\in L^{\infty,+}(\Sigma, {\cal B}(\Sigma), \nu)$ can be shown as well; for example, $D^+(R)$ is given by
\begin{align} D^+(R)=\ell_{\max}\nu^*(\Sigma^0)+\ell_{\min}\nu^*(\Sigma_0)+\sum_{k=1}^r\ell(\Sigma_k)\nu^*(\Sigma_k),\label{countmp}\end{align}
\noi where the optimal probabilities are given by
\begin{subequations}
\label{countall3}
\begin{align}
& \nu^*(\Sigma^0)=\mu(\Sigma^0)+\alpha,\label{countall3a}\\
& \nu^*(\Sigma_0)=\left(\mu(\Sigma_0)-\alpha\right)^+,\label{countall3b}\\
& \nu^*(\Sigma_k)=\left(\mu(\Sigma_k)-\left(\alpha-\sum_{j=1}^k\mu(\Sigma_{j-1})\right)^+\right)^+, \label{countall3c}\\
&\alpha=\min\left(\frac{R}{2},1-\mu(\Sigma^0)\right),\label{countall3d}
\end{align}
\end{subequations}
\noi $k$ is at most countable. We outline the main steps of the derivation. For any $n\in\mathbb{N}$, $\ell \in BC^+(\Sigma)$ define $\ell_n\triangleq\ell\bigwedge n$ (i.e., the minimum between $\ell$ and $n$), then $\ell_n\in BC^+(\Sigma)$, and for any $\nu\in {\mathbb B}_R({ \mu})$ we have
\bes
 \sup_{\nu\in {\mathbb B}_R({ \mu})} \int_{\Sigma}\ell_n(x)d\nu(x)=\frac{R}{2} \Big( \sup_{x \in \Sigma}\ell_n(x) - \inf_{x \in \Sigma} \ell_n(x)\Big) +\int_{\Sigma}\ell_n(x)\nu(dx).
\ees
For any $\nu\in {\mathbb B}_R({ \mu})$, we obtain the inequality
\begin{align*}
\int_{\Sigma}\ell(x)d\nu(x)&=\sup_{n\in \mathbb{N}}\int_{\Sigma}\ell_n(x)\nu(dx)\\
  &\leq \sup_{n\in \mathbb{N}} \sup_{\nu\in {\mathbb B}_R({ \mu})}\int_{\Sigma}\ell_n(x)\nu(dx)\\
  &=\sup_{n\in\mathbb{N}}\left\{\frac{R}{2}\left(\sup_{x\in\Sigma}\ell_n(x)-\inf_{x\in\Sigma}\ell_n(x)\right)+\int_{\Sigma}\ell_n(x)d\mu_n(x)\right\} \\
  &\leq \sup_{n\in\mathbb{N}}\left\{\frac{R}{2}\left(\sup_{x\in\Sigma}\ell_n(x)-\inf_{x\in\Sigma}\ell_n(x)\right)\right\}+\int_{\Sigma}\ell(x)d\mu(x).
\end{align*}
\noi Hence,
\begin{align*}
  \sup_{\nu\in{\mathbb B}_R({ \mu})}\int_{\Sigma}\ell(x)d\nu(x)\leq \frac{R}{2}\sup_{n\in\mathbb{N}}\left\{\sup_{x\in\Sigma}\ell_n(x)-\inf_{x\in\Sigma}\ell_n(x)\right\}+\int_{\Sigma}\ell(x)d\mu(x).
\end{align*}

\noi Similarly, we can show that
\begin{align*}
  \sup_{\nu\in{\mathbb B}_R({ \mu})}\int_{\Sigma}\ell(x)d\nu(x)\geq \frac{R}{2}\sup_{n\in \mathbb{N}}\left\{\sup_{x\in\Sigma}\ell_n(x)-\inf_{x\in\Sigma}\ell_n(x)\right\}+\int_{\Sigma}\ell(x)d\mu(x).
\end{align*}

\noi Hence,
\begin{align*}
  \sup_{\nu\in{\mathbb B}_R({ \mu})}\int_{\Sigma}\ell(x)d\nu(x)=\frac{R}{2}\sup_{n\in\mathbb{N}}\left\{\sup_{x\in\Sigma}\ell_n(x)-\inf_{x\in\Sigma}\ell_n(x)\right\}+\int_{\Sigma}\ell(x)d\mu(x).
\end{align*}

\noi Utilizing the fact that $\sup_{n\in \mathbb{N}} \sup_{x \in \Sigma} \ell_n={\sup_n} ||\ell_n||_{\infty,\nu}$ ( $||\ell||_{\infty,\nu}=\inf_{\Delta\in N}\sup_{x\in\Delta^c}\ell(x)$, $N\triangleq\{A\in {\cal B}(\Sigma):\nu(A)=0\}$, and similarly for the infimum) we obtain the results.

\end{remark}

%
%
%
%
\section{Relation of Total Variational Distance to Other Metrics}
\label{sec.rel.metr}

\noi In this section, we discuss relations of the total variational distance to other distance metrics. We also refer to some applications with distance metrics that can be substituted by the total variational distance metric.

\noi \emph{$L_1$ Distance Uncertainty.} Let $\sigma\in{\cal M}_1(\Sigma)$ be a fixed measure (as well as $\mu\in{\cal M}_1(\Sigma)$). Define the Radon-Nykodym derivatives $\psi\tri\frac{d\mu}{d\sigma}$, $\varphi\tri\frac{d\nu}{d\sigma}$ (densities with respect to a fixed $\sigma\in{\cal M}_1(\Sigma)$). Then,
\vspace{-0.2cm}
\bes
||\nu-\mu||_{TV}=\int|\varphi(x)-\psi(x)|\sigma(dx) \; .
\ees
\noi Consider a subset of ${\mathbb B }_R({ \mu})$ defined by ${\mathbb B }_{R,\sigma}(\mu)\tri\{\nu\in{\mathbb B }_R({ \mu}):\nu<<\sigma,\mu<<\sigma\}\subseteq{\mathbb B }_R({ \mu})$. Then,
\vspace{-0.2cm}
\bes
{\mathbb B }_{R,\sigma}(\mu)=\left\{\varphi\in L_1(\sigma),\varphi\geq 0, \sigma-a.s.:\int_\Sigma|\varphi(x)-\psi(x)|\sigma(dx)\leq R\right\} \; .
\ees

\noi Thus, under the absolute continuity of measures the total variational distance reduces to $L_1$ distance. Robustness via $L_1$ distance uncertainty on the space of spectral densities is investigated in the context of Wiener-Kolmogorov theory in an estimation and decision framework in \cite{Poor80,Vastola84}. The extremum problem described under (a) can be applied to abstract formulations of minimax control and estimation, when the nominal system and uncertainty set are described by spectral measures with respect to variational distance.

\noi {\it Relative Entropy Uncertainty Model.} \cite{Dupuis97} The relative entropy of $\nu\in{\cal M}_1(\Sigma)$ with respect to $\mu\in{\cal M}_1(\Sigma)$ is a mapping $H(\cdot|\cdot):{\cal M}_1(\Sigma)\times{\cal M}_1(\Sigma)\longmapsto[0,\infty]$ defined by
\bes
  H(\nu|\mu) \tri \left\{
  \begin{array}{l l}
    \int_\Sigma\log(\frac{d\nu}{d\mu})d\nu, & \quad \mbox{if}\hso \nu<<\mu\\
    +\infty,  & \quad\mbox{otherwise}.\\
  \end{array} \right.
\ees

\noi It is well known that $H(\nu|\mu)\geq0,\forall\nu,\mu\in {\cal M}_1(\Sigma)$, while $H(\nu|\mu)=0\Leftrightarrow\nu=\mu$. Total variational distance is bounded above by relative entropy via Pinsker's inequality giving
\bea \label{Pinskers}||\nu-\mu||_{TV}\leq \sqrt{2H(\nu|\mu)},\hso \nu,\mu\in{\cal M}_1(\Sigma).\eea

\noi Given a known or nominal probability measure $\mu\in{\cal M}_1(\Sigma)$ the uncertainty set based on relative entropy is defined by
$A_{{\tilde R}}(\mu)\tri\left\{\nu\in{\cal M}_1(\Sigma):H(\nu|\mu)\leq{\tilde R}\right\}$, where ${\tilde R}\in[0,\infty)$. Clearly, the uncertainty set determined by the total variation distance $d_{TV}$, is larger than that determined by the relative entropy. In other words, for every $r>0$, in view of Pinsker's inequality (\ref{Pinskers}):
\begin{align*} \left\{\nu\in{\cal M}_1(\Sigma),\nu<<\mu:H(\nu|\mu)\leq\frac{r^2}{2}\right\}\subseteq {\mathbb B }_R({ \mu})\equiv
\bigg\{\nu\in {\cal M}_1(\Sigma):||\nu-\mu||_{TV}\leq r\bigg\}.\end{align*}
\noi Hence, even for those measures which satisfy $\nu<<\mu$, the uncertainty set described by relative entropy is a subset of the much larger total variation distance uncertainty set. Moreover, by Pinsker's inequality, distance in total variation of probability measures is a lower bound on their relative entropy or Kullback-Leibler distance, and hence convergence in relative entropy of probability measures implies their convergence in total variation distance.

Over the last few years, relative entropy uncertainty model has received particular attention due to various properties (convexity, compact level sets), its simplicity and its connection to risk sensitive pay-off, minimax games, and large deviations \cite{pra96,Ugrinovskii,Petersen,nc2007,Charalambous07}. Recently, an uncertainty model along the spirit of Radon-Nikodym derivative is employed in \cite{Oksendal11} for portfolio optimization under uncertainty. Unfortunately, relative entropy uncertainty modeling has two disadvantages. 1) It does not define a true metric on the space of measures; 2) relative entropy between two measures is not defined if the measures are not absolutely continuous. The latter rules out the possibility of measures $\nu\in{\cal M}_1(\Sigma)$ and $\mu\in{\cal M}_1(\Sigma)$, $\tilde{\Sigma}\subset\Sigma$ to be defined on different spaces\footnote{This corresponds to the case in which the nominal system is a simplified version of the true system and is defined on a lower dimension space.}. It is one of the main disadvantages in employing relative entropy in the context of uncertainty modelling for stochastic controlled diffusions (or SDE's) \cite{pinsker64}. Specifically, by invoking a change of measure it can be shown that relative entropy modelling allows uncertainty in the drift coefficient of stochastic controlled diffusions, but not in the diffusion coefficient, because the latter kind of uncertainty leads to measures which are not absolutely continuous with respect to the nominal measure \cite{pra96}.

\noi {\it Kakutani-Hellinger Distance.} \cite{gibbs} Another measure of distance of two probability measures which relates to their distance in variation is the Kakutani-Hellinger distance. Consider as before, $\nu\in {\cal M}_1(\Sigma)$, $\mu\in {\cal M}_1(\Sigma)$ and a fixed measure $\sigma\in {\cal M}_1(\Sigma)$ such that $\nu<<\sigma$, $\mu<<\sigma$ and define $\varphi\tri\frac{d\nu}{d\sigma}$, $\psi\tri\frac{d\mu}{d\sigma}$. The Kakutani-Hellinger distance is a mapping $d_{KH}:L_1(\sigma)\times L_1(\sigma)\mapsto[0,\infty)$ defined by
\bea \label{Kakutani-Hellinger}d^2_{KH}(\nu,\mu)\tri\frac{1}{2}\int\left(\sqrt{\varphi(x)}-\sqrt{\psi(x)}\right)^2d\sigma(x).\eea

\noi Indeed, the function $d_{KH}$ given by (\ref{Kakutani-Hellinger}) is a metric on the set of probability measures. A related quantity is the Hellinger integral of measures $\nu\in{\cal M}_1(\Sigma)$ and $\mu\in{\cal M}_1(\Sigma)$ defined by \bea H(\nu,\mu)\tri\int \sqrt{\varphi(x)\psi(x)}d\sigma(x),\eea
\noi which is related to the Kakutani-Hellinger distance via $d^2_{KH}(\nu,\mu)=1-H(\nu,\mu)$. The relations between distance in variation and Kakutani-Hellinger distance (and Hellinger integral) are given by the following inequalities:
\begin{align} 2\{1-H(\nu,\mu)\}\leq&||\nu-\mu||_{TV}\leq\sqrt{8\{1-H(\nu,\mu)\}},\\
&||\nu-\mu||_{TV}\leq 2\sqrt{1-H^2(\nu,\mu)},\\
2d^2_{KH}(\nu,\mu)\leq&||\nu-\mu||_{TV}\leq\sqrt{8}d_{KH}(\nu,\mu).
\end{align}
\noi The above inequalities imply that these distances define the same topology on the space of probability measure on $(\Sigma,{\cal B}(\Sigma))$. Specifically, convergence in total variation of probability measures defined on a metric space $(\Sigma,{\cal B}(\Sigma),d)$, implies their weak convergence with respect to the Kakutani-Hellinger distance metric, \cite{gibbs}. In \cite{Ferrante08}, the Hellinger distance on the space of spectral densities is used to define a pay-off subject to constraints in the context of approximation theory.

\noi {\it Levy-Prohorov Distance.} \cite{Dupuis97} Given a metric space $(\Sigma,{\cal B}(\Sigma),d)$, and a family of probability measures ${\cal M}_1(\Sigma)$ on $(\Sigma,{\cal M}_1(\Sigma))$ it is possible to "metrize" weak convergence of probability measure, denoted by $P_n\overset{w}\rightarrow P$, where $\{P_n:n\in{\mathbb N}\}\subset{\cal M}_1(\Sigma)$, $P\in{\cal M}_1(\Sigma)$ via the so called Levy-Prohorov metric denoted by $d_{LP}(\nu,\mu)$. Thus, this metric is also a candidate for a measure of proximity between two probability measures.
\noi The Levi-Prohorov metric is related to distance in variation via the upper bound \cite{gibbs},
\bes d_{LP}(\nu,\mu)\leq\min\left\{||\nu-\mu||_{TV},1\right\}, \hso \forall \hso \nu\in{\cal M}_1(\Sigma),\mu\in{\cal M}_1(\Sigma).\ees
\noi The function defined by $L(\nu,\mu)=\max\left\{d_{LP}(\nu,\mu),d_{LP}(\mu,\nu)\right\}$, is actually a distance metric (it satisfies the properties of distance).

In view of the relations between different metrics, such as relative entropy, Levy-Prohorov metric, Kakutani-Hellinger metric, it is clear that the Problem discussed under (1)-(4) give sub-optimal solution to the same problem with distance in variation replaced by these metrics.

%
%
%
%
\section{Examples} \label{ex}

We will illustrate through simple examples how the optimal solution of the different extremum problems behaves. In particular, we present calculations through Example \ref{exA} for $D^+(R)$ and $R^+(D)$, when the sequence $\ell=\{\ell_1\hso \ell_2\hso \hdots \hso\ell_n\}\in\mathbb{R}_+^n$ consists of a number of $\ell_i$'s which are equal and calculations through Example \ref{exB} for $R^-(D)$ and $D^-(R)$ when the $\ell_i$'s are not equal. We further present calculations through Example \ref{exC} for $D^+(R)$, $R^+(D)$ and $D^-(R)$, $R^-(D)$ using a large number of $\ell_i$'s.

\subsection{Example A}
\label{exA}
\noi Let $\Sigma=\{i:i=1,2,\ldots,8 \}$ and for simplicity consider a descending sequence of lengths $\ell=\{\ell\in \mathbb{R}_+^8:\ell_1=\ell_2>\ell_3=\ell_4>\ell_5>\ell_6=\ell_7>\ell_8\}$ with corresponding nominal probability vector $\mu\in{\mathbb P}_1(\Sigma)$. Specifically, let $\ell=\left[1, 1, 0.8, 0.8, 0.6, 0.4, 0.4, 0.2\right]$, and $\mu=\left[\frac{23}{72} ,\frac{13}{72} ,\frac{10}{72} \,\frac{9}{72} ,\frac{8}{72} ,\frac{4}{72} ,\frac{3}{72},\frac{2}{72}\right]$. Note that, the sets which correspond to the maximum, minimum and all the remaining lengths are equal to $\Sigma^0=\{1,2\}, \Sigma_0=\{8\}, \Sigma_1=\{7,6\}, \Sigma_2=\{5\}, \Sigma_3=\{4,3\}$.
\noi Figures \ref{fig1}\subref{fig1.1}-\subref{fig1.2} depicts the maximum linear functional pay-off subject to total variational constraint, $D^+(R)$, and the optimal probabilities, both given by Theorem \ref{thmproblm3.1}. Figures \ref{fig1}\subref{fig1.3}-\subref{fig1.4} depicts the maximum total variational pay-off subject to linear functional constraint, $R^+(D)$, and the optimal probabilities, both given by Theorem \ref{thmdualproblm3.1}. Recall Lemma \ref{properties} case 1 and Corollary \ref{cor.of.lemma}. Figure \ref{fig1.1} shows that, $D^+(R)$ is a non-decreasing concave function of $R$ and also that is constant in $[R_{\max},2]$, where $R_{\max}=2\left(1-\mu(\Sigma^0)\right)=1$.

\begin{figure}[h!]
\centering
\subfloat[][]{
\label{fig1.1} 
\includegraphics[width=.35\linewidth]{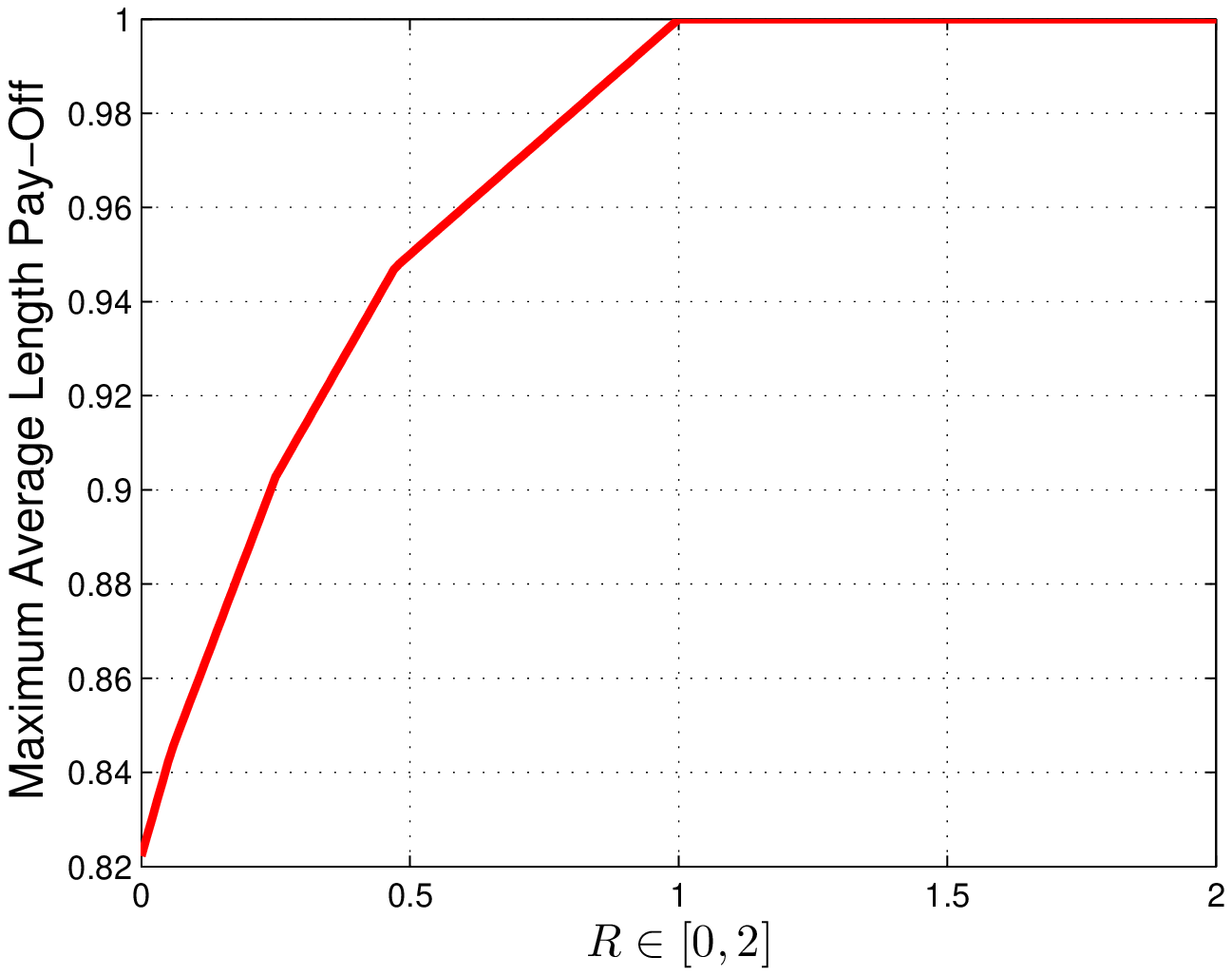}}
\subfloat[][]{
\label{fig1.2} 
\includegraphics[width=.35\linewidth]{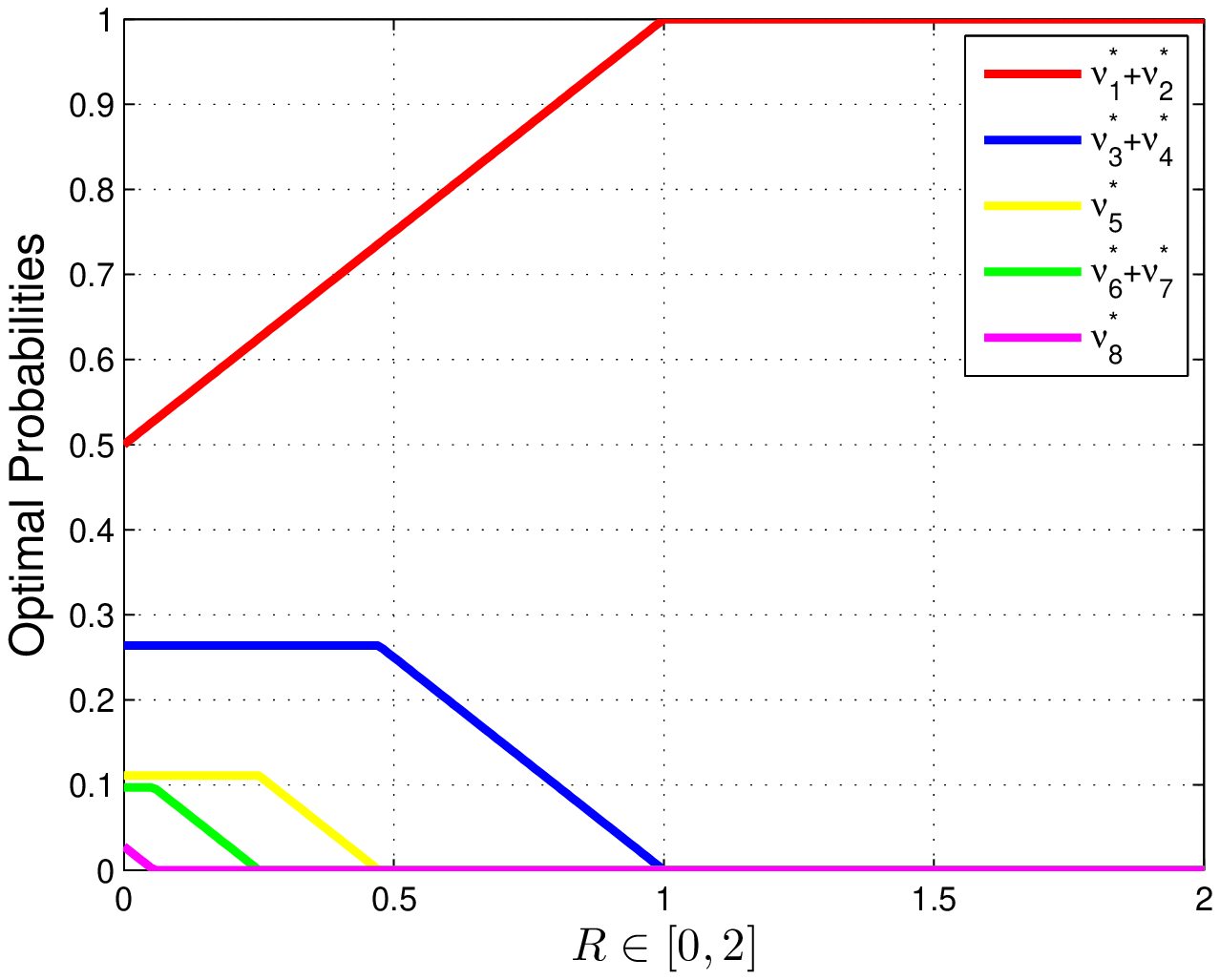}}\\
\subfloat[][]{
\label{fig1.3} 
\includegraphics[width=.35\linewidth]{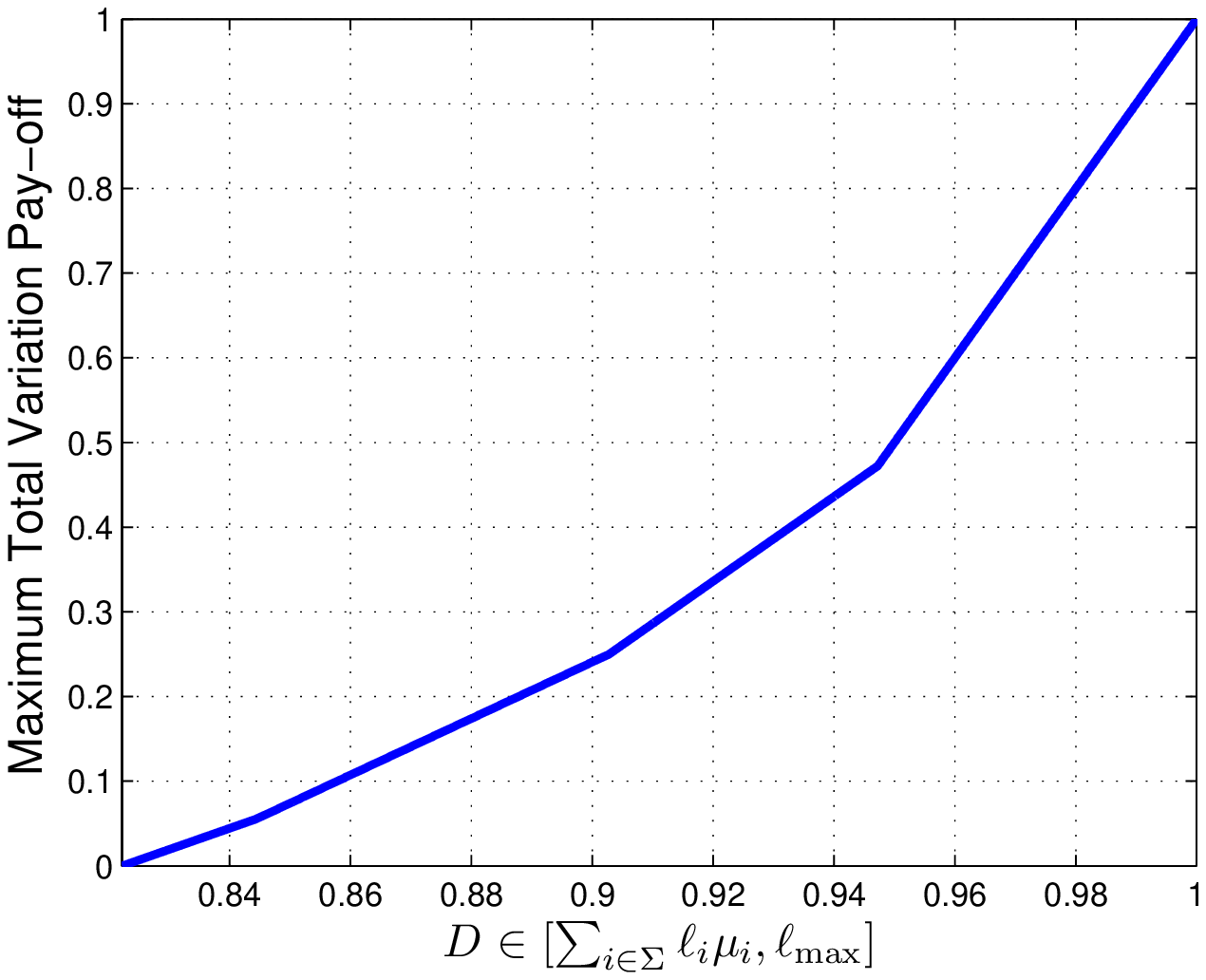}}
\subfloat[][]{
\label{fig1.4} 
\includegraphics[width=.35\linewidth]{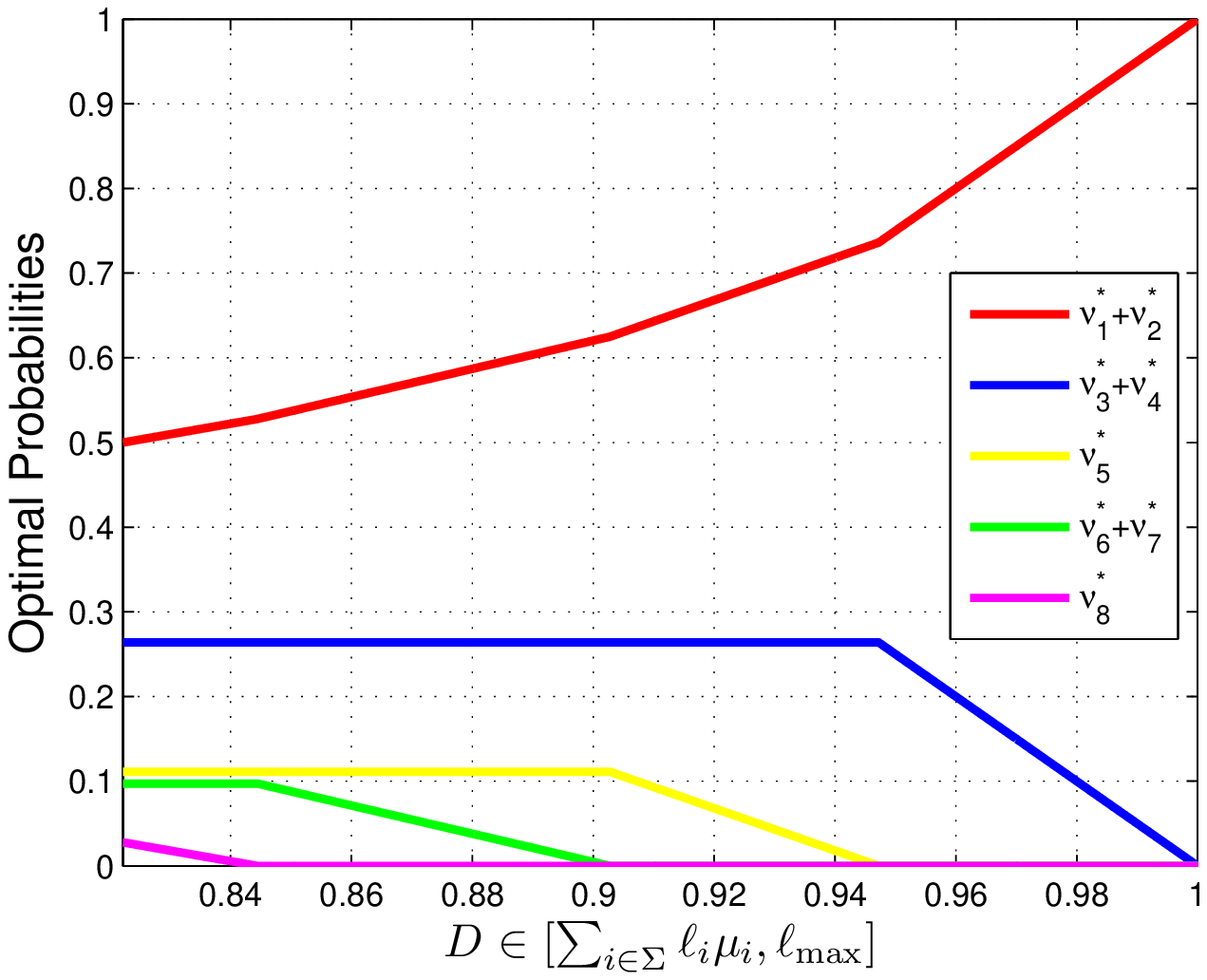}}
\caption[Optimal Solution of Example A]{Solution of Example A: (a) Optimum linear functional pay-off subject to total variational constraint, $D^+(R)$; (b) Optimal probabilities of $D^+(R)$; (c) Optimum total variational pay-off subject to linear functional constraint, $R^+(D)$; and, (d) Optimal probabilities of $R^+(D)$.}
\label{fig1}
\end{figure}

\subsection{Example B}
\label{exB}
\noi Let $\Sigma=\{i:i=1,2,\ldots,8 \}$ and for simplicity consider a descending sequence of lengths $\ell=\{\ell\in\mathbb{R}_+^8:\ell_1>\ell_2>\ell_3>\ell_4>\ell_5>\ell_6>\ell_7>\ell_8 \} $ with corresponding nominal probability vector $\mu\in{\mathbb P}_1(\Sigma)$. Specifically, let $\ell=\left[1,0.8, 0.7, 0.6, 0.5, 0.4, 0.3, 0.2\right]$ and $\mu=\left[\frac{23}{72} ,\frac{13}{72} ,\frac{10}{72} ,\frac{9}{72} ,\frac{8}{72} ,\frac{4}{72} ,\frac{3}{72} ,\frac{2}{72}\right]$. Note that, the sets which correspond to the maximum, minimum and all the remaining lengths are equal to $\Sigma^0=\{1\}, \Sigma_0=\{8\}, \Sigma^1=\{2\}, \Sigma^2=\{3\}, \Sigma^3=\{4\}, \Sigma^4=\{5\}, \Sigma^5=\{6\}, \Sigma^6=\{7\}$.
\noi Figures \ref{fig2}\subref{fig2.1}-\subref{fig2.2} depicts the minimum total variational pay-off subject to linear functional constraint, $R^-(D)$, and the optimal probabilities, both given by Theorem \ref{thmproblm3.2}. Figures \ref{fig2}\subref{fig2.3}-\subref{fig2.4} depicts the minimum linear functional pay-off subject to total variational constraint, $D^-(R)$, and the optimal probabilities, both given by Theorem \ref{thmdualproblm3.2}.
\begin{figure}[h!]
\centering
\subfloat[]{
\label{fig2.1} 
\includegraphics[width=.35\linewidth]{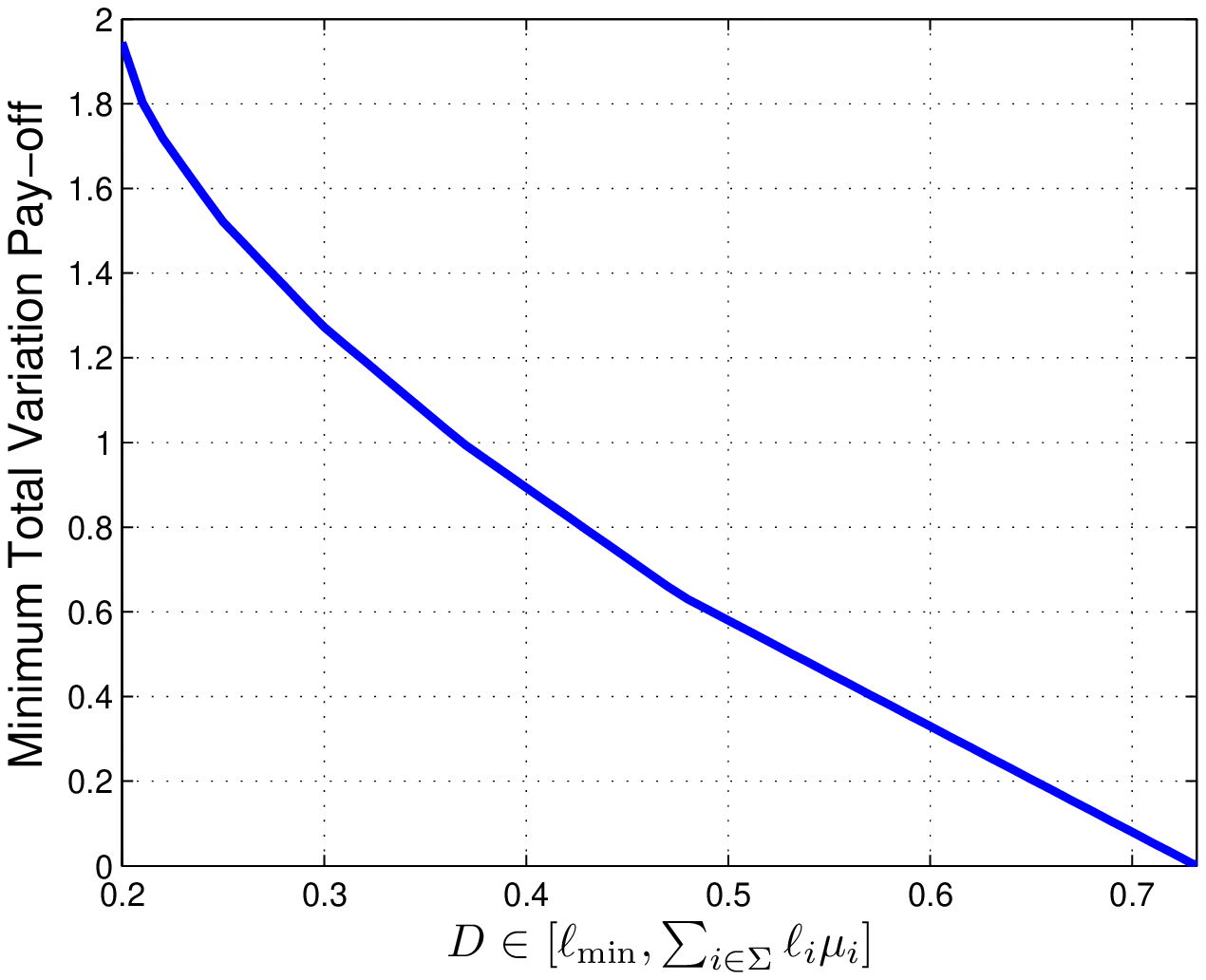}}
\subfloat[]{
\label{fig2.2} 
\includegraphics[width=.35\linewidth]{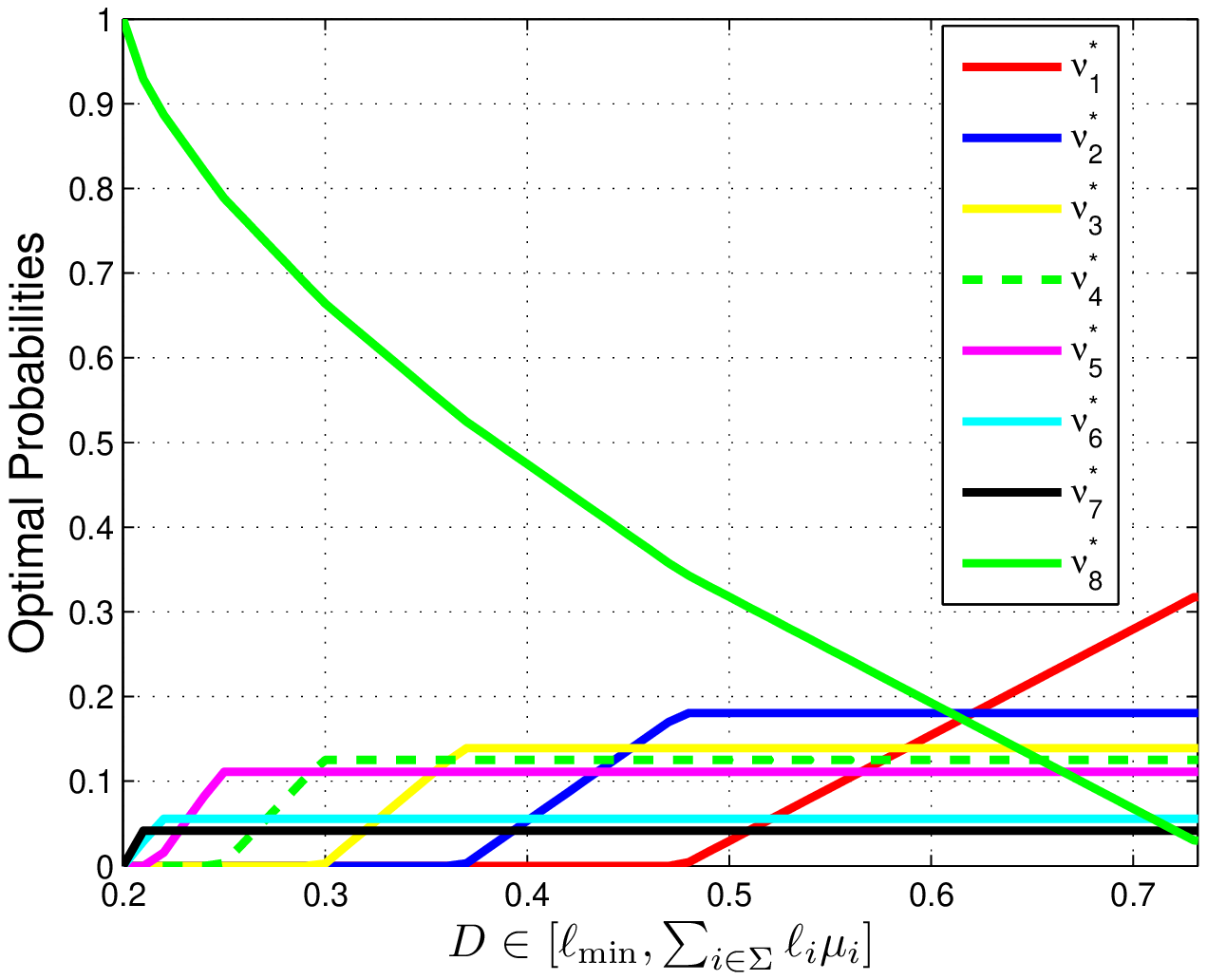}}\\
\subfloat[]{
\label{fig2.3} 
\includegraphics[width=.35\linewidth]{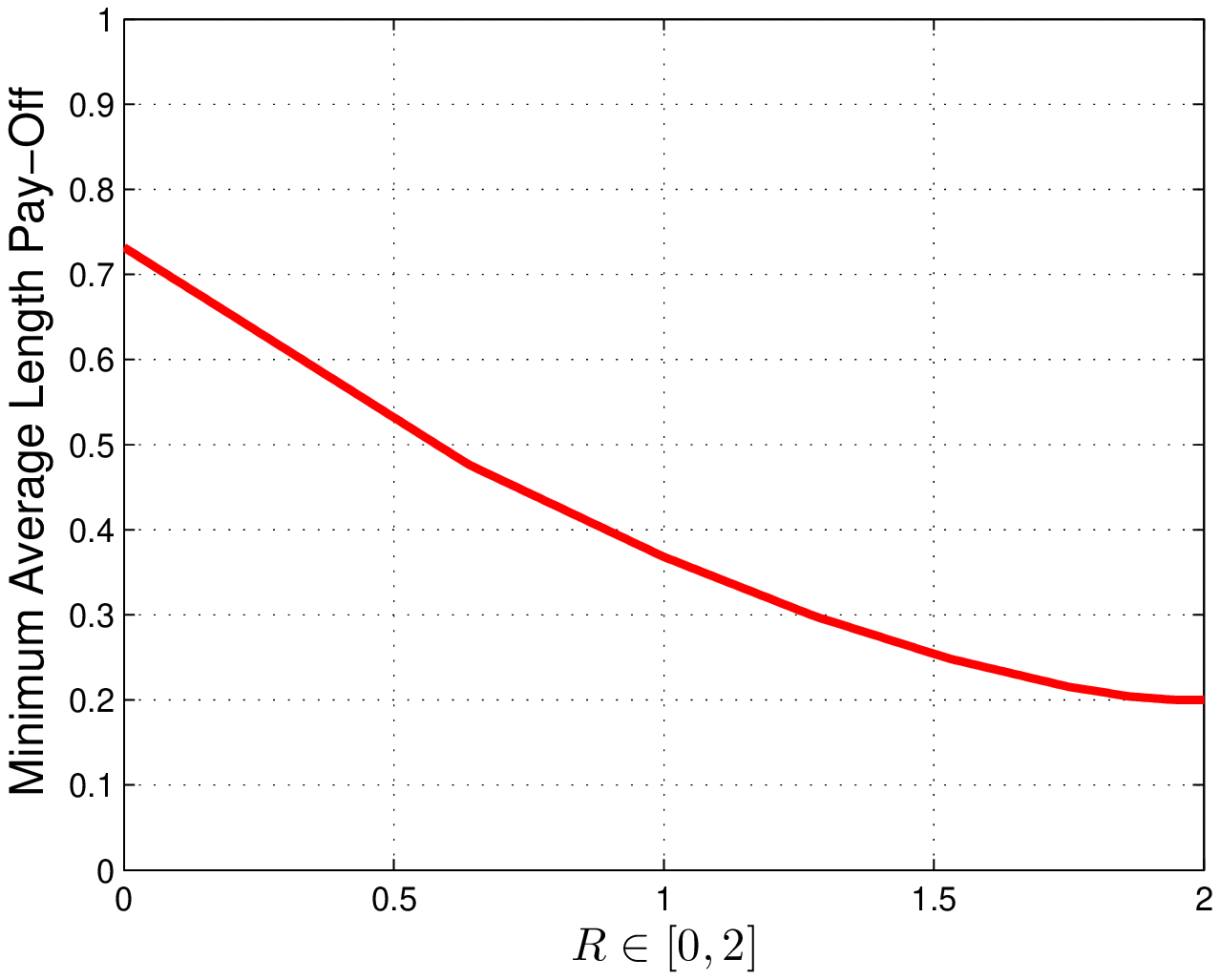}}
\subfloat[]{
\label{fig2.4} 
\includegraphics[width=.35\linewidth]{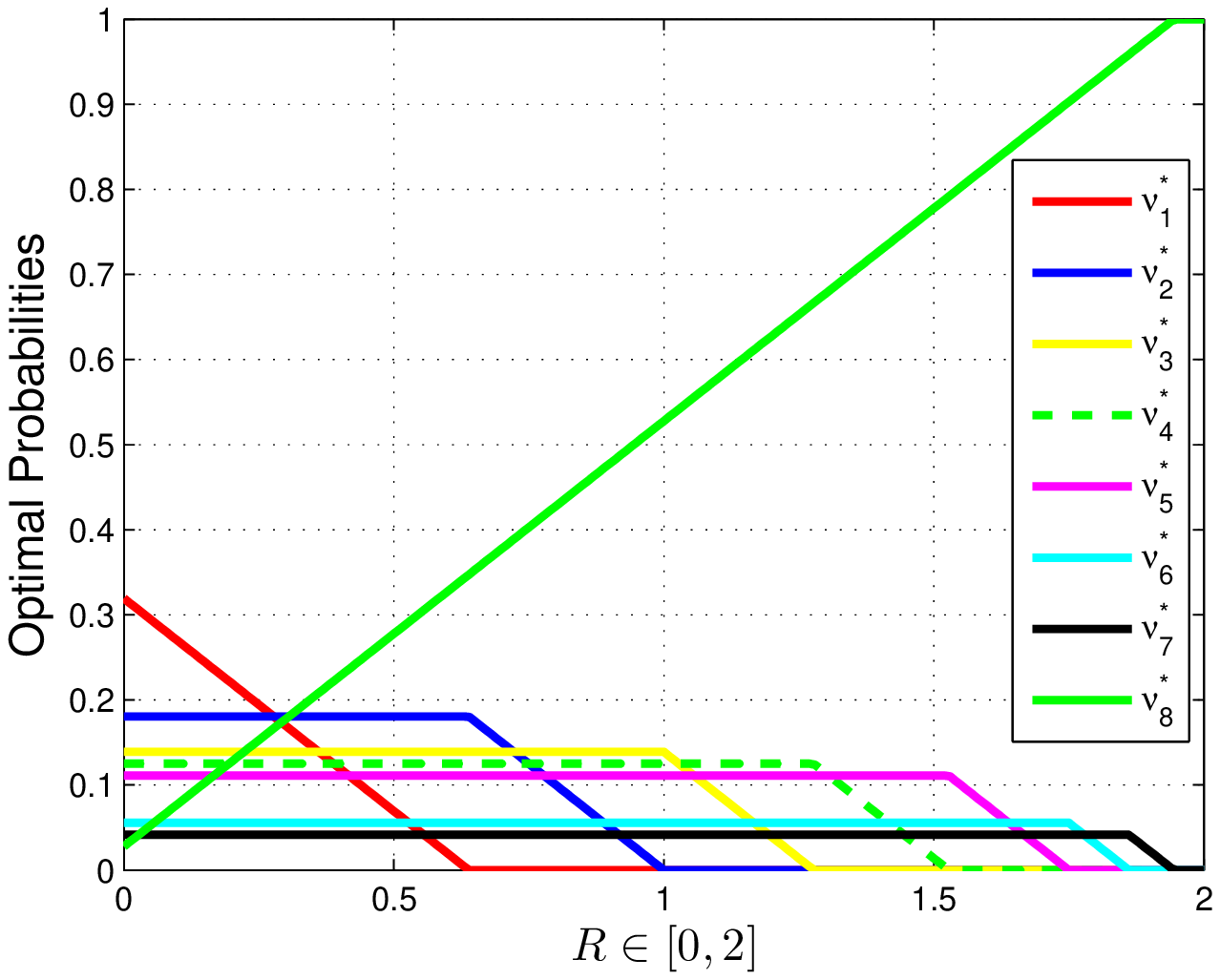}}
\caption[Optimal Solution of Example B]{Solution of Example B: (a) Optimum total variational pay-off subject to linear functional constraint, $R^-(D)$; (b) Optimal probabilities of $R^-(D)$; (c) Optimum linear functional pay-off subject to total variational constraint, $D^-(R)$; and, (d) Optimal probabilities of $D^-(R)$.}
\label{fig2}
\end{figure}

\noi Recall Lemma \ref{properties} case 2 and Corollary \ref{cor.of.lemma}. Figure \ref{fig2.1} shows that, $R^-(D)$ is a non-increasing convex function of $D$, $D\in[\ell_{\min},\sum_{i\in\Sigma}\ell_i\mu_i)$. Note that for $D<\ell_{\min}=0.2$ no solution exists and $R^-(D)$ is zero in $[D_{\max},\infty)$ where $D_{\max}=\sum_{i=1}^8\ell_i\mu_i=0.73$.

\subsection{Example C}
\label{exC}
\noi Let $\Sigma=\{i:i=1,2,\ldots,50 \}$ and consider a descending sequence of lengths $\ell=\{\ell\in\mathbb{R}_+^{50} \} $ with corresponding nominal probability vector $\mu\in{\mathbb P}_1(\Sigma)$. For display purposes the support sets are denoted by $\Sigma_x^y$ where $x,y=\{1,2,\hdots,16\}$, though of course the subscript symbol $x$ corresponds to the support sets of Problem $D^+(R)$, $R^+(D)$ and the superscript symbol $y$ corresponds to the support sets of Problem $D^-(R)$ and $R^-(D)$. Let
\begin{multline}
\ell=\Big[20 \ \ 20 \ \ 20 \ \ 20 \ \ 19 \ \ 19 \ \ 19 \ \ 18 \ \ 17 \ \ 17 \ \ 16 \ \ 14 \ \ 14\ \ 13 \ \ 13 \ \ 13 \ \ 13 \ \ 12 \ \ 10 \ \ 10 \ \ 10 \ \ 10\\
\shoveleft{10 \ \ 9 \ \ 9 \ \ 9\ \ 8\ \ 8\ \ 8\ \ 8\ \ 8\ \ 8 \ \ 8 \ \ 7 \ \ 7 \ \ 6 \ \ 5 \ \ 4 \ \ 3 \ \ 3 \ \ 3 \ \ 3 \ \ 3\ \ 3 \ \ 2\ \ 2 \ \ 2 \ \ 2 \ \ 1 \ \ 1 \Big]},\nonumber
\end{multline}

\noi and
\begin{multline}
\mu=\Big[ 0.052 \ \ 0.002 \ \  0.01\ \  0.006 \ \  0.004 \ \  0.038 \ \  0.032 \ \  0.028\ \ 0.026 \ \ 0.008 \ \ 0.012 \ \  0.01 \ \  0.008\\
\shoveleft{  0.026 \ \  0.05 \ \  0.044 \ \ 0.03 \ \ 0.032 \ \  0.024 \ \  0.01 \ \ 0.02 \ \ 0.03 \ \ 0.014 \ \  0.024 \ \ 0.004 \ \ 0.006 \ \ 0.024}\\
\shoveleft{ 0.01 \ \ 0.022 \ \ 0.012 \ \  0.016 \ \ 0.042 \ \ 0.014 \ \ 0.016 \ \ 0.01\ \ 0.024 \ \ 0.02 \ \ 0.008 \ \ 0.014 \ \ 0.032 \ \  0.018  }\\
\shoveleft{0.012 \ \ 0.01 \ \ 0.04\ \  0.036 \ \ 0.018 \ \ 0.002 \ \ 0.022 \ \ 0.012 \ \ 0.016\Big]}.\nonumber
\end{multline}

\noi Note that, the sets which correspond to the maximum, minimum and all the remaining lengths are equal to
\begin{align*} &\Sigma^0=\{1-4\},\Sigma_0=\{50,49\},\Sigma_1^{16}=\{48-45\},\Sigma_2^{15}=\{44-39\},\Sigma_3^{14}=\{38\},\Sigma_4^{13}=\{37\},\\
&\Sigma_5^{12}=\{36\},\Sigma_6^{11}=\{35,34\},\Sigma_7^{10}=\{33-27\},\Sigma_8^9=\{26-24\},\Sigma_9^8=\{23-19\},\Sigma_{10}^7=\{18\},\\
&\Sigma_{11}^6=\{17-14\},\Sigma_{12}^5=\{13,12\},\Sigma_{13}^4=\{11\},\Sigma_{14}^3=\{10-9\},\Sigma_{15}^2=\{8\},\Sigma_{16}^1=\{7-5\}.
\end{align*}
\noi Figures \ref{fig3}\subref{fig3.1}-\subref{fig3.2} depicts the maximum linear functional pay-off subject to total variational constraint, $D^+(R)$, and the maximum total variational pay-off subject to linear functional constraint, $R^+(D)$, given by Theorem \ref{thmproblm3.1}, \ref{thmdualproblm3.1}, respectively. Figures \ref{fig3}\subref{fig3.3}-\subref{fig3.4} depicts the minimum linear functional pay-off subject to total variational constraint, $D^-(R)$, and the minimum total variational pay-off subject to linear functional constraint, $R^-(D)$, given by Theorem \ref{thmdualproblm3.2}, \ref{thmproblm3.2} respectively.

\begin{figure}[h!]
\centering
\subfloat[]{
\label{fig3.1} 
\includegraphics[width=.35\linewidth]{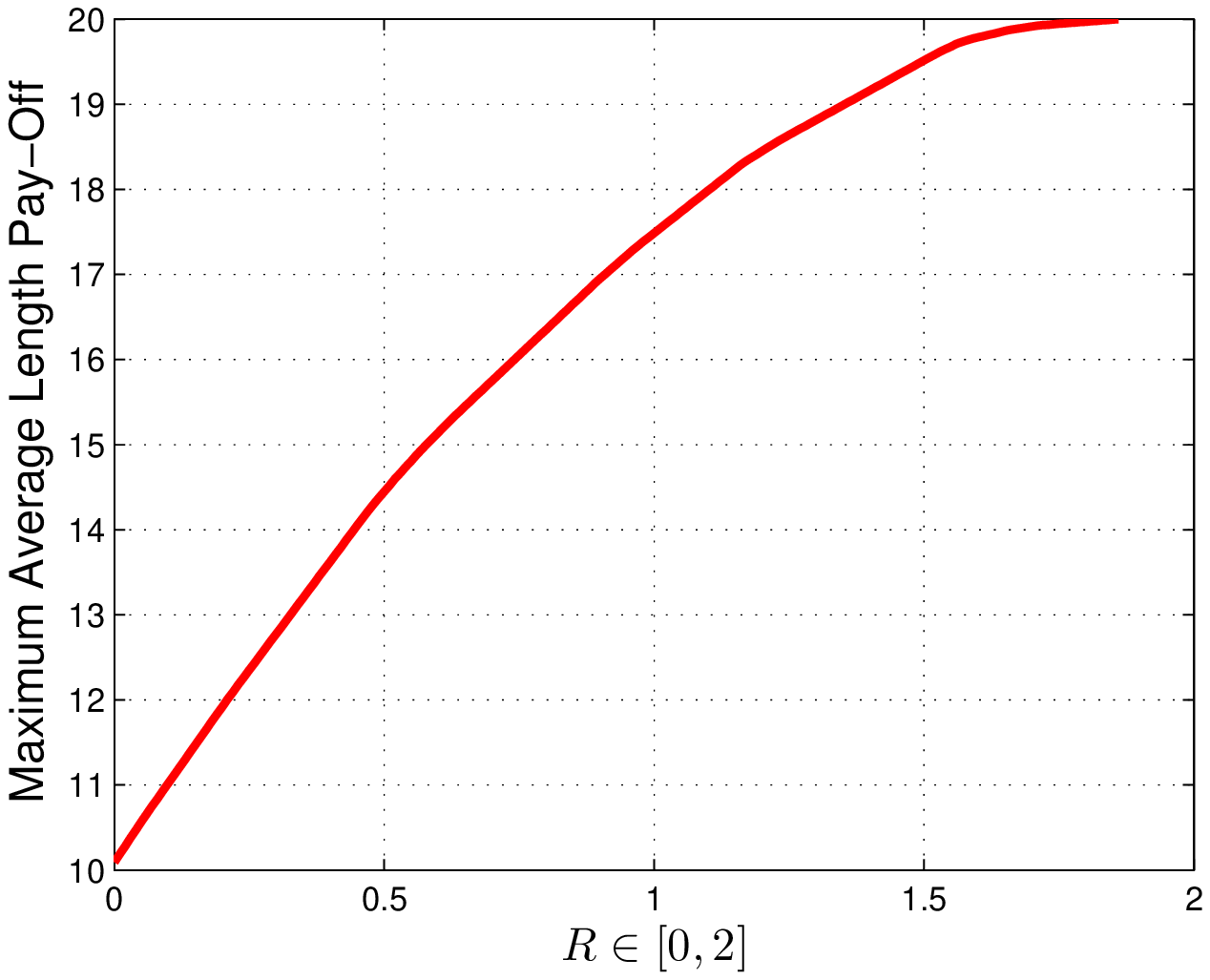}}
\subfloat[]{
\label{fig3.2} 
\includegraphics[width=.35\linewidth]{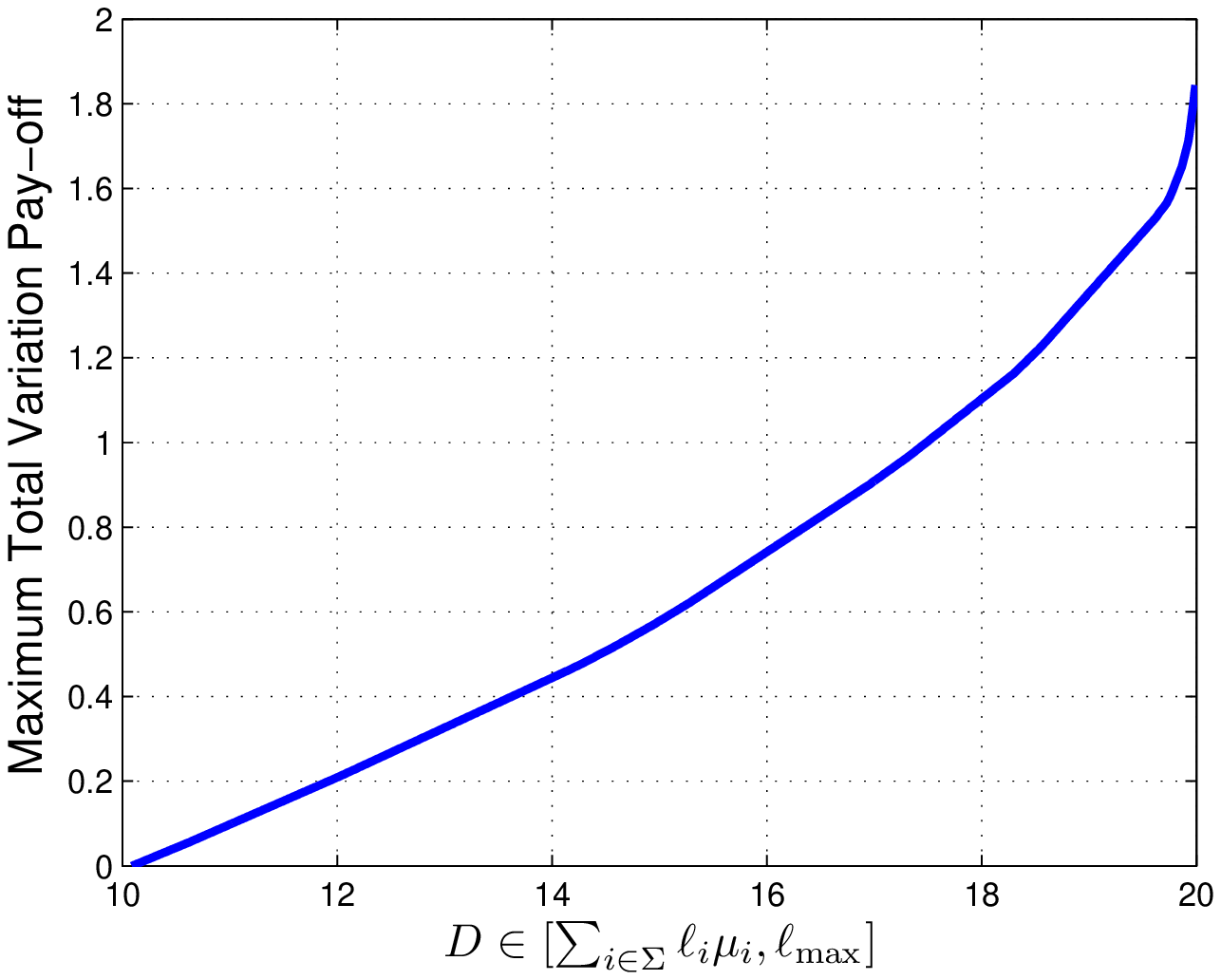}}\\
\subfloat[]{
\label{fig3.3} 
\includegraphics[width=.35\linewidth]{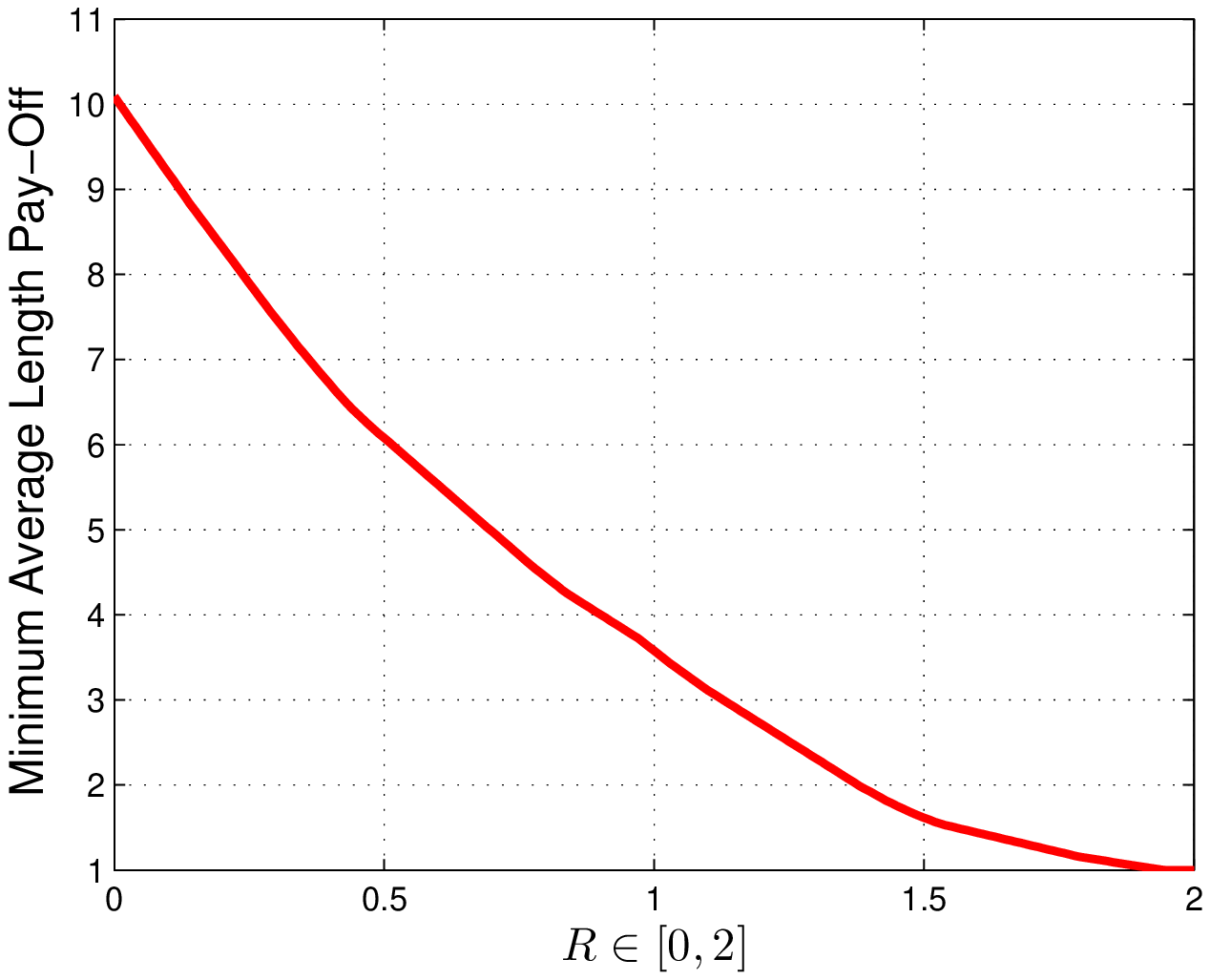}}
\subfloat[]{
\label{fig3.4} 
\includegraphics[width=.35\linewidth]{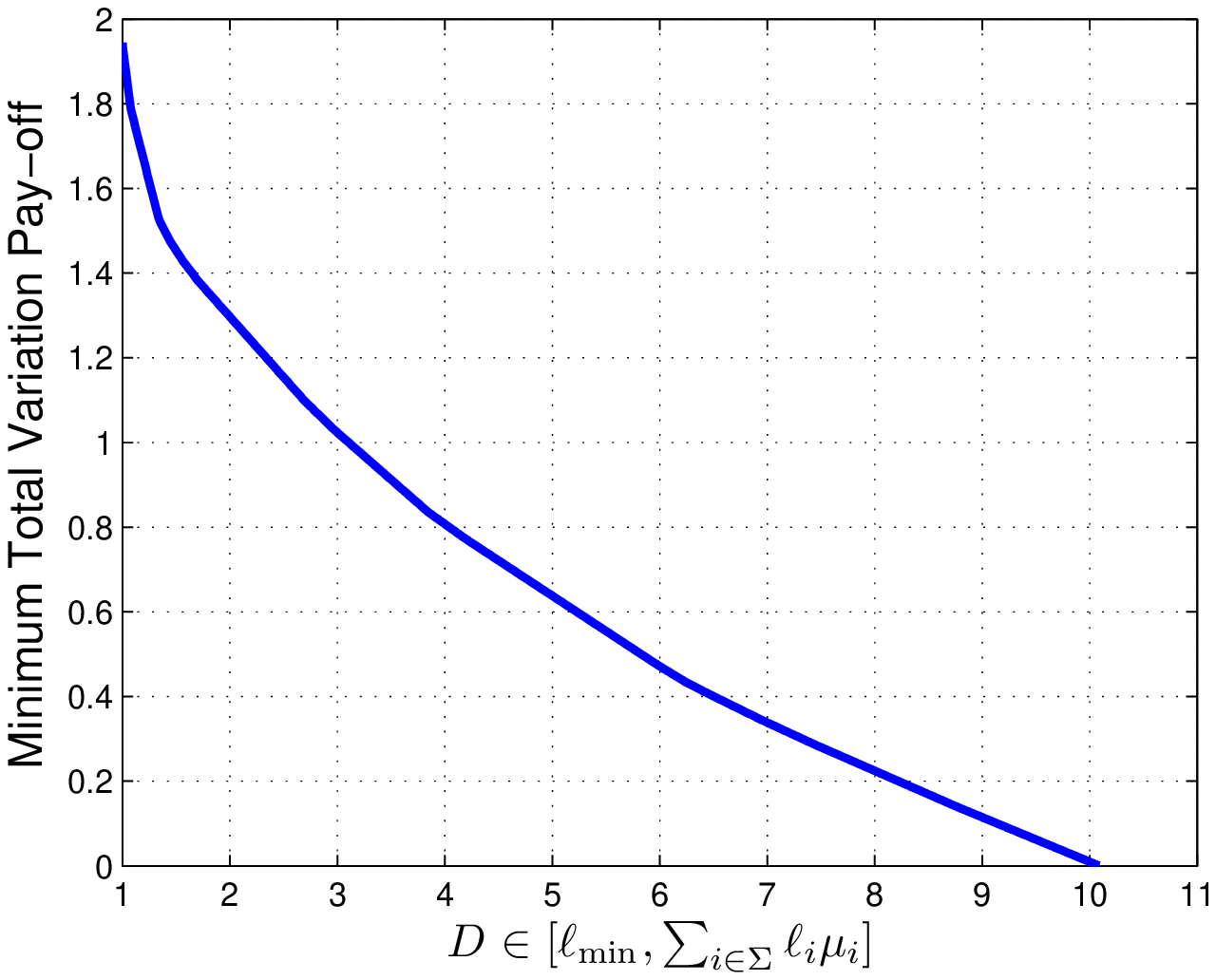}}
\caption[Optimal Solution of Example C]{Solution of Example C: (a) Optimum linear functional pay-off subject to total variational constraint, $D^+(R)$; (b) Optimum total variational pay-off subject to linear functional constraint, $R^+(D)$; (c) Optimum linear functional pay-off subject to total variational constraint, $D^-(R)$; and, (d) Optimum total variational pay-off subject to linear functional constraint, $R^-(D)$.}
\label{fig3}
\end{figure}

%
%
%
%
\section{Conclusion}
\label{sec.conclusions}
\noi This paper is concerned with extremum problems involving total variational distance metric as a pay-off subject to linear functional constraints, and vice-versa; that is, with the roles of total variational metric and linear functional interchanged. These problems are formulated using concepts from signed measures while the theory is developed on abstract spaces. Certain properties and applications of the extremum problems are discussed, while closed form expressions of the extremum measures are derived for finite alphabet spaces. Finally, it is shown through examples how the extremum solution of the various problems behaves. Extremum problems have a wide variety of applications, spanning from Markov decision problems to model reduction.

%
%
%
%
\appendix[Proof of Theorem \ref{thmproblm3.2}]
\label{app1}

\begin{lemma}
\label{alt.dual.min.lemma}
\noi The following bounds hold.\\
1. Lower Bound.
\vspace{-0.3cm}
\bea \sum_{i\in\Sigma}\ell_i\xi_i^+\geq\ell_{\min}\left(\frac{\alpha}{2}\right).
\eea
\noi The bound holds with equality if
\bes
\sum_{i\in\Sigma_0} \mu_i+\frac{\alpha}{2}\leq 1,\hso\sum_{i\in\Sigma_0}\xi_i^+=\frac{\alpha}{2},\hso \xi_i^+=0\hso \mbox{for}\hso i\in\Sigma\setminus\Sigma_0,
\ees
\noi and the optimal probability on $\Sigma_0$ is given by \bes \nu^*(\Sigma_0)\triangleq\sum_{i\in\Sigma_0}\nu_i^*=\min\left(1,\sum_{i\in\Sigma_0}\mu_i+\frac{\alpha}{2}\right).\ees

\noi 2. Upper Bound.
\vspace{-0.3cm}
\bea \sum_{i\in\Sigma}\ell_i\xi_i^-\leq\ell_{\max}\left(\frac{\alpha}{2}\right).
\eea
\noi The bound holds with equality if
\vspace{-0.1cm}
\bes
\sum_{i\in \Sigma^0}\mu_i-\frac{\alpha}{2}\geq 0,\hso\sum_{i\in \Sigma^0}\xi_i^-=\frac{\alpha}{2},\hso \xi_i^-=0\hso\mbox{for}\hso i\in \Sigma\setminus\Sigma^0,
\ees
\noi \noi and the optimal probability on $\Sigma^0$ is given by
\vspace{-0.3cm}
\bes \nu^*(\Sigma^0)\triangleq\sum_{i\in\Sigma^0}\nu_i^*=\left(\sum_{i\in\Sigma^0}\mu_i-\frac{\alpha}{2}\right)^+.\ees
\end{lemma}

\begin{proof}
\noi Follows from Section \ref{subsec.Equiv.Extr.Prob}.
\end{proof}

\begin{proposition}\label{propminprblm}
\noi If $\sum_{i\in\Sigma_0}\mu_i+\frac{\alpha}{2}=1$  and $\nu_i^*\geq \mu_i$ for all $i\in\Sigma_0$ then $R^-(D)=2(1-\sum_{i\in\Sigma_0}\mu_i)$.
\end{proposition}

\begin{proof}
\noi The condition $\sum_{i\in\Sigma_0}\mu_i+\frac{\alpha}{2}=1$ implies that $\sum_{i\in\Sigma_0}\nu_i^*=1$ and therefore $\sum_{i\in\Sigma\setminus\Sigma_0}\nu_i^*=0$, hence $\nu_i^*=0$, for all $i\in\Sigma\setminus\Sigma_0$. Then the minimum pay-off (\ref{dualmp}) is given by
\begin{align*}
R^-(D)&=\sum_{i\in\Sigma_0}\left|\nu_i^*-\mu_i\right|+\sum_{i\in\Sigma\setminus\Sigma_0}\left|\nu_i^*-\mu_i\right|
=\sum_{i\in\Sigma_0}\left|\nu_i^*-\mu_i\right|+\sum_{i\in\Sigma\setminus\Sigma_0}\left|-\mu_i\right|\\
&\overset{(a)}= \sum_{i\in\Sigma_0}\nu_i^*-\sum_{i\in\Sigma_0}\mu_i+\sum_{i\in\Sigma\setminus\Sigma_0}\mu_i=\left(1-\sum_{i\in\Sigma_0}\mu_i\right)+\left(1-\sum_{i\in\Sigma_0}\mu_i\right)
=2\left(1-\sum_{i\in\Sigma_0}\mu_i\right).
\end{align*}
\noi where (a) follows due to the fact that $\nu_i^*\geq \mu_i$ for all $i\in\Sigma_0$.
\end{proof}

\noi Next, we show the derivation of \eqref{eq5.1}.
\begin{lemma}\label{alt.dual.min.subcase1}
\noi Under the conditions of Lemma \ref{alt.dual.min.lemma}, then
\begin{align}
R^-(D)=\frac{\displaystyle 2\left(D-\sum_{i\in\Sigma}\ell_i\mu_i\right)}{\ell_{\min}-\ell_{\max}}.\label{alt.dual.min.subcase1.eq.1}
\end{align}
\end{lemma}

\begin{proof}
\noi From (\ref{alt.dual.min.average.constraint}) and Lemma \ref{alt.dual.min.lemma} we have
\begin{align*}
D\geq\sum_{i\in\Sigma}\ell_i\xi_i^+-\sum_{i\in\Sigma}\ell_i\xi_i^-+\sum_{i\in\Sigma}\ell_i\mu_i=\ell_{\min}\left(\frac{\alpha}{2}\right)-\ell_{\max}\left(\frac{\alpha}{2}\right)+\sum_{i\in\Sigma}\ell_i\mu_i.
\end{align*}

\noi Solving the above equation with respect to $\alpha$ we get that
\begin{align*}
\alpha\leq \frac{\displaystyle 2\left(D-\sum_{i\in\Sigma}\ell_i\mu_i\right)}{\ell_{\min}-\ell_{\max}}.
\end{align*}

\noi If we select the solution on the boundary then, (\ref{alt.dual.min.subcase1.eq.1}) is obtained.
\end{proof}


\begin{corollary}\label{alt.dual.min.subcase3}
\noi For any $k\in\{1,2,\hdots,r\}$ if the following conditions hold
\begin{subequations}
\begin{align}
&\sum_{j=1}^k\sum_{i\in\Sigma^{j-1}}\mu_i-\frac{\alpha}{2}\leq 0\hso \mbox{and}\hso \sum_{j=0}^k\sum_{i\in\Sigma^{j}}\mu_i-\frac{\alpha}{2}\geq 0,\label{opprobcond}\\
&\sum_{i\in\Sigma^{j-1}}\xi_i^-=\sum_{i\in\Sigma^{j-1}}\mu_i,\hso\mbox{for all}\hso j=1,2,\hdots,k,\\
&\sum_{i\in \Sigma^k}\xi_i^-=\left(\frac{\alpha}{2}-\sum_{j=1}^k\sum_{i\in\Sigma^{j-1}}\mu_i\right),\label{opprobsigmk}\\
&\xi_i^-=0 \hso\mbox{for all}\hso i\in\Sigma\setminus\Sigma^0\cup\Sigma^1\cup\hdots\cup\Sigma^k,\\
&\sum_{i\in\Sigma_0}\mu_i+\frac{\alpha}{2}<1,\\
&\sum_{i\in\Sigma_0}\xi_i^+=\frac{\alpha}{2}, \quad \xi_i^+=0 \hso\mbox{for all}\hso i\in\Sigma\setminus\Sigma_0,
\end{align}
\end{subequations}
\noi then
\begin{align}
R^-(D)=\frac{\displaystyle 2\left(D-\ell_{\min}\sum_{i\in\Sigma_0}\mu_i-\ell\left(\Sigma^k\right)\sum_{j=1}^{k}\sum_{i\in\Sigma^{j-1}}\mu_i-\sum_{j=k}^r\sum_{i\in\Sigma^j}\ell_i\mu_i\right)}{\ell_{\min}-\ell\left({\Sigma^k}\right)}.\label{alt.dual.min.subcase3.eq.1}
\end{align}
\noi Moreover, the optimal probability on $\Sigma^k$ is given by \bea \label{opprosigk} \nu^*(\Sigma^k)\triangleq\sum_{i\in\Sigma^k} \nu_i^*=\left(\sum_{i\in\Sigma^k} \mu_i-\left(\frac{\alpha}{2}-\sum_{j=1}^k\sum_{i\in\Sigma^{j-1}}\mu_i\right)^+\right)^+.\eea
\end{corollary}

\begin{proof}
\noi Under the conditions stated, we have that
\begin{align*}
\sum_{i\in\Sigma}\ell_i\xi_i^-&= \sum_{j=1}^k\sum_{i\in\Sigma^{j-1}}\ell_i\xi_i^-+\sum_{i\in\Sigma^k}\ell_i\xi_i^-+\sum_{i\in\Sigma\setminus\displaystyle\cup_{j=0}^k\Sigma^j}\ell_i\xi_i^-=\sum_{j=1}^k\sum_{i\in\Sigma^{j-1}}\ell_i\mu_i+\ell(\Sigma^k)\left(\frac{\alpha}{2}-\sum_{j=1}^k\sum_{i\in\Sigma^{j-1}}\mu_i\right).
\end{align*}
\noi Also,
\vspace{-0.4cm}
\begin{align*}
\sum_{i\in\Sigma}\ell_i\xi_i^+=\sum_{i\in\Sigma_0}\ell_i\xi_i^++\sum_{i\in\Sigma\setminus\Sigma_0}\ell_i\xi_i^+=\ell_{\min}\left(\frac{\alpha}{2}\right).
\end{align*}

\noi From (\ref{alt.dual.min.average.constraint}), we have that
\begin{align*}
D&\geq\sum_{i\in\Sigma}\ell_i\xi_i^+-\sum_{i\in\Sigma}\ell_i\xi_i^-+\sum_{i\in\Sigma}\ell_i\mu_i\\
&=\ell_{\min}\left(\sum_{i\in\Sigma_0}\mu_i+\frac{\alpha}{2}\right)-\sum_{j=1}^k\sum_{i\in\Sigma^{j-1}}\ell_i\mu_i-\ell({\Sigma^k})\left(\frac{\alpha}{2}-\sum_{j=1}^k\sum_{i\in\Sigma^{j-1}}\mu_i\right)+\sum_{i\in\Sigma\setminus\Sigma_0}\ell_i\mu_i.
\end{align*}

\noi Solving the above equation with respect to $\alpha$ we get that
\begin{align*}
\alpha\leq \frac{\displaystyle 2\left(D-\ell_{\min}\sum_{i\in\Sigma_0}\mu_i-\ell(\Sigma^k)\sum_{j=1}^{k}\sum_{i\in\Sigma^{j-1}}\mu_i-\sum_{j=k}^r\sum_{i\in\Sigma^j}\ell_i\mu_i\right)}{\ell_{\min}-\ell(\Sigma^k)}.
\end{align*}
\noi If we select the solution at the boundary then, (\ref{alt.dual.min.subcase3.eq.1}) is obtained. From (\ref{opprobsigmk}) we have that
\bes
\sum_{i\in \Sigma^k}\xi_i^-=\left(\frac{\alpha}{2}-\sum_{j=1}^k\sum_{i\in\Sigma^{j-1}}\mu_i\right),
\text{    and hence,     } \sum_{i\in \Sigma^k}\nu_i=\sum_{i\in \Sigma^k}\mu_i-\left(\frac{\alpha}{2}-\sum_{j=1}^k\sum_{i\in\Sigma^{j-1}}\mu_i\right).\ees
\noi The optimal $\sum_{i\in\Sigma^k}\nu_i^*$ must satisfy (\ref{opprobcond}). Hence, (\ref{opprosigk}) is obtained.
\end{proof}

\noi Putting together Lemma \ref{alt.dual.min.lemma}, Proposition \ref{propminprblm}, Lemma \ref{alt.dual.min.subcase1}, and Corollary \ref{alt.dual.min.subcase3} we obtain the result of Theorem \ref{thmproblm3.2}.

\vspace{-0.3cm}
\bibliographystyle{IEEEtran}
\bibliography{Bibliography}

\begin{thebibliography}{10}
\providecommand{\url}[1]{#1}
\csname url@samestyle\endcsname
\providecommand{\newblock}{\relax}
\providecommand{\bibinfo}[2]{#2}
\providecommand{\BIBentrySTDinterwordspacing}{\spaceskip=0pt\relax}
\providecommand{\BIBentryALTinterwordstretchfactor}{4}
\providecommand{\BIBentryALTinterwordspacing}{\spaceskip=\fontdimen2\font plus
\BIBentryALTinterwordstretchfactor\fontdimen3\font minus
  \fontdimen4\font\relax}
\providecommand{\BIBforeignlanguage}[2]{{%
\expandafter\ifx\csname l@#1\endcsname\relax
\typeout{** WARNING: IEEEtran.bst: No hyphenation pattern has been}%
\typeout{** loaded for the language `#1'. Using the pattern for}%
\typeout{** the default language instead.}%
\else
\language=\csname l@#1\endcsname
\fi
#2}}
\providecommand{\BIBdecl}{\relax}
\BIBdecl

\bibitem{cover}
T.~M. Cover and J.~A. Thomas, \emph{Elements of Information Theory}.\hskip 1em
  plus 0.5em minus 0.4em\relax John Wiley and Sons, Inc., 1991.

\bibitem{Meyn93}
S.~P. Meyn and R.~L. Tweedie, \emph{Markov Chains and Stochastic
  Stability}.\hskip 1em plus 0.5em minus 0.4em\relax London: Springer-Verlag,
  1993.

\bibitem{gibbs}
A.~L. Gibbs and F.~E. SU, ``On choosing and bounding probability metrics,''
  \emph{Internat. Statist. Rev}, vol.~70, no.~3, pp. 419--435, Dec. 2002.

\bibitem{Dupuis97}
P.~Dupuis and R.~S. Ellis, \emph{A Weak Convergence Approach to the Theory of
  Large Deviations}.\hskip 1em plus 0.5em minus 0.4em\relax New York: John
  Wiley \& Sons, Inc., 1997.

\bibitem{cif2007}
C.~D. Charalambous, I.~Tzortzis, and F.~Rezaei, ``Stochastic optimal control of
  discrete-time systems subject to conditional distribution uncertainty,'' in
  \emph{50th IEEE Conference on Decision and Control and European Control
  Conference}, Orlando, Florida, Dec. 12--15, 2011, pp. 6407--6412.

\bibitem{dunford}
N.~Dunford and J.~Schwartz, \emph{Linear Operators: Part 1: General
  Theory}.\hskip 1em plus 0.5em minus 0.4em\relax New York: Interscience
  Publishers, Inc., 1957.

\bibitem{pra96}
P.~D. Pra, L.~Meneghini, and W.~J. Runggaldier, ``Connections between
  stochastic control and dynamic games,'' \emph{Math. Control Signals Systems},
  vol.~9, no.~4, pp. 303--326, 1996.

\bibitem{Ugrinovskii}
V.~A. Ugrinovskii and I.~R. Petersen, ``Finite horizon minimax optimal control
  of stochastic partially observed time varying uncertain systems,''
  \emph{Math. Control Signals Systems}, vol.~12, no.~1, pp. 1--23, 1999.

\bibitem{Petersen}
I.~R. Petersen, M.~R. James, and P.~Dupuis, ``Minimax optimal control of
  stochastic uncertain systems with relative entropy constraints,'' \emph{IEEE
  Trans. Autom. Control}, vol.~45, no.~3, pp. 398--412, Mar. 2000.

\bibitem{nc2007}
N.~U. Ahmed and C.~D. Charalambous, ``Minimax games for stochastic systems
  subject to relative entropy uncertainty: Applications to sde's on hilbert
  spaces,'' \emph{J. Math. Control Signals Systems}, vol.~19, no.~1, pp.
  65--91, Feb. 2007.

\bibitem{Charalambous07}
C.~D. Charalambous and F.~Rezaei, ``Stochastic uncertain systems subject to
  relative entropy constraints: Induced norms and monotonicity properties of
  minimax games,'' \emph{IEEE Trans. Autom. Control}, vol.~52, no.~4, pp.
  647--663, Apr. 2007.

\bibitem{Poor80}
H.~V. Poor, ``On robust wiener filtering,'' \emph{IEEE Trans. Autom. Control},
  vol.~25, no.~3, pp. 531--536, Jun. 1980.

\bibitem{Vastola84}
K.~S. Vastola and H.~V. Poor, ``On robust wiener-kolmogorov theory,''
  \emph{IEEE Trans. Inform. Theory}, vol.~30, no.~2, pp. 315--327, Mar. 1984.

\bibitem{ctc2012}
C.~D. Charalambous, I.~Tzortzis, and T.~Charalambous, ``Dynamic programming
  with total variational distance uncertainty,'' in \emph{51st IEEE Conference
  on Decision and Control}, Maui, Hawaii, Dec. 10--13, 2012.

\bibitem{Wong85}
E.~Wong and B.~Hajek, \emph{Stochastic Processes in Engineering Systems}.\hskip
  1em plus 0.5em minus 0.4em\relax New York: Springer-Verlag, 1985.

\bibitem{Ferrante08}
A.~Ferrante, M.~Pavon, and F.~Ramponi, ``Hellinger vs. kullback-leibler
  multivariable spectrum approximation,'' \emph{IEEE Trans. Autom. Control},
  vol.~53, no.~5, pp. 954--967, May 2008.

\bibitem{georgiou06}
T.~T. Georgiou, ``Relative entropy and the multivariable multidimensional
  moment problem,'' \emph{IEEE Trans. Inform. Theory}, vol.~52, no.~3, pp.
  1052--1066, Mar. 2006.

\bibitem{Ferrante07}
A.~Ferrante, M.~Pavon, and F.~Ramponi, ``Constrained approximation in the
  hellinger distance,'' in \emph{Proceedings of European Control Conference},
  Kos, Greece, Jul. 2--5, 2007, pp. 322--327.

\bibitem{Georgiou03kullback-leiblerapproximation}
T.~T. Georgiou and A.~Lindquist, ``Kullback-leibler approximation of spectral
  density functions,'' \emph{IEEE Trans. Inform. Theory}, vol.~49, no.~11, pp.
  2910--2917, Nov. 2003.

\bibitem{Pavon06}
M.~Pavon and A.~Ferrante, ``On the georgiou-lindquist approach to constrained
  kullback-leibler approximation of spectral densities,'' \emph{IEEE Trans.
  Autom. Control}, vol.~51, no.~4, pp. 639--644, Apr. 2006.

\bibitem{fcn2012}
F.~Rezaei, C.~D. Charalambous, and N.~U. Ahmed, ``Optimal control of uncertain
  stochastic systems subject to total variation distance uncertainty,''
  \emph{SIAM Journal on Control and Optimization}, vol.~50, no.~5, pp.
  2683--2725, Sep. 2012.

\bibitem{jaynes1}
E.~T. Jaynes, ``Information theory and statistical mechanics,'' \emph{Physics
  Review}, vol. 106, pp. 620--630, 1957.

\bibitem{jaynes2}
------, ``Information theory and statistical mechanics ii,'' \emph{Physics
  Review}, vol. 108, pp. 171--190, 1957.

\bibitem{baras}
J.~S. Baras and M.~Rabi, ``Maximum entropy models, dynamic games, and robust
  output feedback control for automata,'' in \emph{Proceedings of the 44th IEEE
  Conference on Decision and Control, and the European Control Conference},
  Seville, Spain, Dec. 12--15, 2005.

\bibitem{rezcharahm12}
F.~Rezaei, C.~D. Charalambous, and N.~U. Ahmed, \emph{Optimization of
  Stochastic Uncertain Systems: Entropy Rate Functionals, Minimax Games and
  Robustness. A Festschrift in Honor of Robert J Elliott}, ser. Advances in
  Statistics, Probability and Actuarial Science.\hskip 1em plus 0.5em minus
  0.4em\relax World Scientific Publishing Company Incorporated, 2012.

\bibitem{halmos}
P.~R. Halmos, \emph{Measure Theory}.\hskip 1em plus 0.5em minus 0.4em\relax
  Springer-Verlag New York Inc., 1974.

\bibitem{Oksendal11}
B.~Oksendal and A.~Sulem, ``Portfolio optimization under model uncertainty and
  bsde games,'' Institut National de Recherche en Informatique et en
  Automatique, Rapport de Recherche 7554, 2011.

\bibitem{pinsker64}
M.~S. Pinsker, \emph{Information and Information Stability of Random Variables
  and Processes}.\hskip 1em plus 0.5em minus 0.4em\relax San Francisco:
  Holden-Day, 1964.

\end{thebibliography}

\end{document}